%% file: main.tex
\documentclass[12pt,final]{amsart}
\usepackage{fancybox}
\usepackage[style=alphabetic,url=false, isbn=false, doi=false,maxbibnames=99]{biblatex}
\usepackage[margin=1in]{geometry}
\usepackage{amsmath}
\usepackage{amsthm}
\usepackage{amssymb}
\usepackage{ytableau,nicefrac,xspace}
\usepackage[capitalize]{cleveref}
\usepackage{stmaryrd}
\usepackage{xcolor}
\usepackage{quiver}
\newtheorem{ithm}{Theorem}

\numberwithin{equation}{section}

\newtheorem{thm}{Theorem}[section]
\newtheorem{conj}[thm]{Conjecture}
\newtheorem{defn}[thm]{Definition}
\newtheorem{lem}[thm]{Lemma}
\newtheorem{prop}[thm]{Proposition}
\newtheorem{rmk}[thm]{Remark}
\newtheorem{cor}[thm]{Corollary}
\newtheorem{question}[thm]{Question}
\usepackage{tikz,tikz-cd}
\theoremstyle{remark}

\Crefname{thm}{Theorem}{Theorems}

\newtheorem{example}[thm]{Example}
\newcommand{\Hh}{\mathcal{H}}
\newcommand{\Sl}{\mathcal{L}}
\newcommand{\Sd}{\mathcal{S}}
\newcommand{\Sym}{\mathfrak{S}}
\newcommand{\res}{\operatorname{res}}
\newcommand{\un}{\mathcal{U}}

\newcommand{\st}{\mathcal{S}}
\newcommand{\std}{\mathcal{T}}
\usepackage[nomargin,inline]{fixme}
\makeatletter
\renewcommand*\FXLayoutInline[3]{%
  {\@fxuseface{inline}
  \ignorespaces\noindent \ovalbox{\hspace{.01\textwidth} \begin{minipage}{.95\textwidth}
  	#3 \fxnotename{#1}: #2
  \end{minipage}\hspace{.01\textwidth}}}
  \newline}
\makeatother

\newcommand{\homh}{\operatorname{Hom}_{\Hh_n}}

\newcommand{\End}{\operatorname{End}}
\newcommand{\Ext}{\operatorname{Ext}}
\newcommand{\Endh}{\operatorname{End}_{\mathcal{H}_n}}
\newcommand{\Ind}[2]{\operatorname{Ind}_{#1}^{#2}}
\newcommand{\Res}[2]{\operatorname{Res}_{#2}^{#1}}
\newcommand{\hres}[2]{{}^{\mathcal{H}\!}\operatorname{Res}_{#2}^{#1}}
\newcommand{\hind}[2]{{}^{\mathcal{H}\!}\operatorname{Ind}_{#1}^{#2}}

\newcommand{\la}{\lambda}
\newcommand{\al}{\alpha}
\newcommand{\e}{e_n}
\newcommand{\eh}[1]{e^{\mathcal{H}}_{#1}}
\newcommand{\Ch}[1]{C_{\Hh_n}^{#1}}
\newcommand{\Csl}[1]{C_{\Sl}^{#1}}
\newcommand{\Csd}[1]{C_{\Sd}^{#1}}
\newcommand{\Lh}[1]{\overline{N}_{#1}}

\newcommand{\Lsd}[1]{\overline{N}_{\Sd}^{#1}}
\newcommand{\Lnm}{\Lambda_{n,m}} 
\newcommand{\Lpnm}{\Lambda_{n,m}^+}
\newcommand{\Sh}[1]{D^{#1}_{\Hh}}
\newcommand{\Ssd}[1]{D^{#1}_{\Sd}}
\newcommand{\Ssl}[1]{D^{#1}_{\Sl}}
\newcommand{\teeny}[1]{\mbox{\fontsize{4pt}{6pt}\selectfont $#1$}}
\newcommand{\oml}[1]{\Phi_{#1}} 
\newcommand{\Hom}{\operatorname{Hom}}
\newcommand{\F}[1]{\Sigma_{#1}}

\newcommand{\mmod}{\operatorname{-mod}}
\newcommand{\Simp}{\mathsf{Simp}}
\newcommand{\emm}[2]{\mathsf{e}(#1,#2)}
\newcommand{\K}{\Bbbk}
\newcommand{\Z}{\mathbb{Z}}

\newcommand{\LNX}[1]{LNX$_{{#1}}$\xspace}
\newcommand{\LnmB}{\Lambda_{n,m}^{B}}
\newcommand{\bweyl}[1]{W_{#1}}
\newcommand{\Anm}{B_n(m)}

\newcommand{\cO}{\mathcal{O}}
\newcommand{\C}{\mathbb{C}}

\newcommand{\KZ}{\mathsf{KZ}}
\newcommand{\pol}[1]{h_{#1}} 
\newcommand{\cell}[1]{W^{#1}} 
\title[Representation theory of Schur algebras in type B]{On the representation theory of\\ Schur algebras in type B}
\date{}
\author{Dinushi Munasinghe} 
\address[D.M.]{Department of Mathematics, University of Toronto 
Toronto, ON, Canada}
\email{dinushi.munasinghe@mail.utoronto.ca}

\author{Ben Webster}
\address[B.W.]{Department of Pure Mathematics, University of Waterloo \& Perimeter Institute for Theoretical Physics,
Waterloo, ON, Canada}
\email{ben.webster@uwaterloo.ca}

\begin{document}
\begin{abstract}
    We study the representation theory of the type B Schur algebra $\Sl^n(m)$ with unequal parameters introduced by Lai and Luo.  For generic values of $q,Q$, this algebra is semi-simple and Morita equivalent to the Hecke algebra, but for special values, its category of modules is more complicated.  We study this representation theory by comparison with the cyclotomic $q$-Schur algebra of Dipper, James and Mathas, and use this to construct a cellular algebra structure on $\Sl^n(m)$.

    This allows us to index the simple $\Sl^n(m)$-modules as a subset of the set of bipartitions of $n$.  For $m$ large, this will be all bipartitions of $n$ if and only if $\Sl^n(m)$ is quasi-hereditary, in which case, $\Sl^n(m)$ is Morita equivalent to the cyclotomic $q$-Schur algebra.  We prove a modified version of a conjecture of Lai, Nakano and Xiang giving the values of $(Q,q)$ where this holds: if $m$ is large and odd, $Q\neq -q^k$ for all $k$ satisfying $\frac{4-n}{2}\leq k<n$; if $m$ is large and even, $Q\neq -q^k$ for all $k$ satisfying $-n<k<n.$  We also prove two strengthenings of this result: an indexing of the simple modules when $q$ is not a root of unity, and a characterization of the quasi-hereditary blocks of $\Sl^n(m)$.
\end{abstract}
\maketitle
\ytableausetup{centertableaux}
\section{Introduction}

A fundamental object in representation theory is the Schur algebra, which arises as the algebra of $\Sym_n$-invariant endomorphisms of $V^{\otimes n}$ via Schur--Weyl duality. In 1986, Jimbo generalized this duality to the quantum case \cite{Jimbo}, considering the commutant of the type A Hecke algebra $\mathcal{H}_n^A(q)$. This lead to the construction of the $q$-Schur algebra, introduced by Dipper and James in 1989 \cite{DJ}.  For any base field $\Bbbk$ and parameter $q\in \Bbbk^{\times}$, this algebra is a cellular, quasi-hereditary cover of $\mathcal{H}_n^A(q)$, mirroring the relationship of the classical Schur algebra to the group algebra of the symmetric group in characteristic $p$,
and can be constructed equivalently as homomorphisms over permutation modules or as $\mathcal{H}_n^A(q)$-invariant endomorphisms on the tensor space. 

However, when generalizing to type B, these two perspectives lead to different notions of a type B Schur algebra.  One approach is to try to preserve a version of Schur--Weyl duality and compatibility with induction from parabolic subgroups.  This approach is natural if one studies singular Soergel bimodules, as in the work of Williamson \cite{Williamson}. This was expanded to the unequal parameter case by Lai and Luo \cite{lai2019schur} and Lai, Nakano, and Xiang \cite{LNX}, building on the work of Green \cite{green97} and Bao, Kujawa, Li, and Wang \cite{BKLW}.
In this perspective, one constructs a type B Schur algebra $\Sl^n(m)$ which is the commutant of an action of the type B Hecke algebra $\mathcal{H}_n^B(Q,q)$ for parameters $q,Q \in \Bbbk^{\times}$ acting on $(\Bbbk^m)^{\otimes n}$.   Like the $q$-Schur algebra in type A, the dependence of this algebra on $m$ is rather small---it can be complicated for low values of $m$, but if $m\geq 2n$, the Morita equivalence class of $\Sl^n(m)$ only depends on the parity of $m$.  Since some results only apply in this case, we say $m$ is {\bf large} if $m\geq 2n$, and {\bf large odd} or {\bf large even} depending on the parity of $m$.  

Another approach, which has proven very influential in representation theory, is to preserve being a quasi-hereditary cover.  It is also known how to do this---one can consider a type B Hecke algebra as one example of a cyclotomic Hecke algebra.  By work of Dipper, James, and Mathas \cite{DJM}, 
the category of modules over the {\bf cyclotomic $q$-Schur algebra} $\Sd^n(\Pi_n)$ provides a highest weight cover of the module category $\mathcal{H}_n^B(Q,q)\mmod$ when $\Pi_n$ is the set of all bipartitions of $n$.   As with the dependence of $\Sl^n(m)$ on $m$, the cyclotomic $q$-Schur algebra depends on a set $\Lambda$ of bicompositions, though up to Morita equivalence, only the set of bipartitions in $\Lambda$ matters.  The most frequent choice we will make is $\Lambda =\Lnm$, the set of bicompositions where the first component has $n$ parts and the second $\lfloor \nicefrac{m}{2}\rfloor$.  If $m$ is large, then all bipartitions lie in $\Lnm$ and $\Sd^n(\Pi_n)$ and $\Sd^n(\Lnm)$ are Morita equivalent.  

\begin{rmk}
    We will be working with the unequal parameters case. All type B algebras will depend on two parameters, $q$ and $Q$, where $\{q, -1\}$  and $\{Q, -1\}$ are the roots of the quadratic relation for reflection in long and short roots, respectively. As we will not change these parameters often, we will suppress this notation and write $\Hh_n$ for the type B Hecke algebra $\Hh_n^B(Q,q)$ and similarly suppress $q,Q$ in the notation of $\Sd^n$ and $\Sl^n(m)$. 
\end{rmk}

Our aim in this paper is to compare the representation theory of the algebras $\Sl^n(m)$ and $\Sd^n$.  The category of modules over $\Sl^n$ has been investigated in \cite{LNX} and over $\Sd$ in \cite{DJM}, showing that for generic values of $q,Q$, both algebras are Morita equivalent to $\mathcal{H}_n$, and in fact to the group algebra of the Weyl group $W_n=W_{B_n}$.  However, for some special values of $q$ and $Q$, these algebras will not be Morita equivalent, and it is an interesting question how their categories of modules are different.  

The key connection between the representation theory of Hecke and Schur algebras is the Schur functor $\Omega_n\colon\Sd^n(\Pi_n)\mmod \to \mathcal{H}_n\mmod$, which we can realize as the functor $M\mapsto M\eh{n}$ for an idempotent satisfying $\eh{n}\Sd^n(\Pi_n)\eh{n}=\mathcal{H}_n$.  That is, this functor realizes $\mathcal{H}_n\mmod$ as a quotient of $\Sd^n(\Pi_n)\mmod$ by the subcategory of modules killed by $\eh{n}$. 

We will show that there is a similar relation between $\Sd^n(\Lnm)$ and $\Sl^n(m)$, and that the quotient functor to $\mathcal{H}_n\mmod$ factors through the category $\Sl^n(m)\mmod$ for $m$ large.  That is:
\begin{ithm}\label{th:A}
We have an isomorphism $\Sl^n(m)\cong \e\Sd^n(\Lnm)\e$ for an idempotent $\e$. 
Thus, we have a quotient functor $\Sd^n(\Lnm)\mmod \to \Sl^n(m) \mmod$. For $m$ large, $\Omega_n$ factors through this quotient:
   \[\Sd^n(\Lnm)\mmod \to \Sl^n(m) \mmod \to\mathcal{H}_n\mmod. \]
\end{ithm}

This result has some very important consequences for the category $\Sl^n(m)\mmod$.  In particular, the algebra $\Sl^n(m)$ is cellular and so the simple modules of $\Sl^n(m)$ are canonically indexed by a subset of bipartitions of $n$, depending on the parameters $Q,q$; we say that a bipartition is {\bf \LNX{m}} if it lies in this set.  In other words, a bipartition $\lambda$ of $n$ is \LNX{m} if $e_n\Ssd{\lambda} \neq 0,$ where $\Ssd {\lambda}$ is the simple of $\Sd^n(\Lnm)$ indexed by this bipartition.

We say that a bipartition of $n$ is \LNX{o} if it is \LNX{m} for large odd $m$ and \LNX{e} if it is \LNX{m} for large even $m$.  
Kleshchev bipartitions are automatically \LNX{o} and \LNX{e}, but for many choices of $q,Q,n$, there will be other bipartitions that satisfy one or both of these conditions.  For small values of $m$, the situation is even more complicated and even Kleshchev bipartitions might not be \LNX{m}.  

Like the Kleshchev bipartitions, the \LNX{o} and \LNX{e} bipartitions form a sub-crystal of all bipartitions, as we show using induction and restriction functors on the algebras $\Sl^n(m)$ and $\Sl^{n+1}(m+2)$.   Based on a few explicit calculations, 
this allows us to show, with some modifications to the original statement, a conjecture of Lai, Nakano, and Xiang \cite[Conj. 6.1.3]{LNX}:  
\begin{ithm}\label{th:B}
Assume $m$ is large.  The following are equivalent: 
\begin{enumerate}
    \item The algebra $\Sl^n(m)$ is Morita equivalent to $\Sd^n(\Lnm)$.
    \item All bipartitions of $n$ are \LNX{m}.  
    \item The algebra $\Sl^n(m)$ is quasi-hereditary
    \item If $m$ is large odd, $Q\neq -q^k$ for all $k$ satisfying $\frac{4-n}{2}\leq k<n$; if $m$ is large even, $Q\neq -q^k$ for all $k$ satisfying $-n<k<n.$
\end{enumerate}
    If $q$ is not a root of unity of order $e \leq n$, then a bipartition  $(\lambda^{(1)},\lambda^{(2)})$ is \LNX{o} if and only if it is in the crystal graph component of $((k),\emptyset)$ for some $k$, and \LNX{e} if and only if it is Kleshchev.   
\end{ithm} 
The hypothesis that $q$ is not a root of unity of order $e \leq n$ is sharp, since if this fails,  then the bipartition $(\emptyset, (n))$ is not Kleshchev, but this bipartition is \LNX{m} for all $m$.  We expect that this part of this result will generalize to the case where $q$ is arbitrary if we incorporate Heisenberg crystal operations as in the work of Gerber and Norton \cite{gerberMathfrakslInfty2018}. 

We also characterize the blocks of $\Sl^n(m)$ and find which are quasi-hereditary in terms of inclusions of Young diagrams (\cref{prop:blocks-qh}).  


Finally, we discuss some natural connections to the representation theory of rational Cherednik algebras.  In type A, Rouquier showed that category $\cO$ for a Cherednik algebra is equivalent to modules over a type A Schur algebra of the same rank \cite{RouqSchur}.  It's natural to ask how this result extends in type B. Lai, Nakano and Xiang showed that for generic values of $Q$, we have a similar equivalence \cite[Th. 8.3.3]{LNX}; however, \cref{th:B} makes it clear this cannot hold for all $q,Q$, since this category $\cO$ is always highest weight.  Similarly, we can see a problem from the perspective of the rational Cherednik algebra: this algebra depends on a choice of parameters $h,H$, and we have a functor $\KZ\colon \cO^n_{h,H}\to \Hh_n^B\mmod$ where $q=e^{2\pi i h},Q=e^{2\pi i H}$.  Of course, many  different choices of $h,H$, will give the same $q,Q$, and amongst these, we can obtain inequivalent categories. In this case, they cannot all be equivalent to $\Sl^n(m)\mmod$.  

Instead, we have an analogue of \cref{th:A}---a quotient functor $\KZ^B\colon \cO^n_{h,H}\to \Sl^n(m)\mmod$ when $\Re h\leq 0$ and $ \Re H\leq 0$, which is a modification of the functor $\KZ$.  This functor is an equivalence if and only if $\Sl^n(m)$ is quasi-hereditary. 

While this quotient functor is not necessarily an equivalence, it does have a nice relationship to Rouquier's work on uniqueness of covers---if $m$ is large odd, the functor $\KZ^B$ is a 0-faithful cover and thus becomes 1-faithful after suitable deformation.    This shows that we can reconstruct the category $\cO^n_{h,H}$ purely from the category $\Sl^n(m)\mmod$ for $m$ large odd and the order on bipartitions that corresponds to the quasi-hereditary order.

\section*{Acknowledgments}

The authors acknowledge the support of the Sydney Mathematical Research Institute and University of Sydney, where part of this research was carried out.  In particular, we appreciate the hospitality of Andrew Mathas, Oded Yacobi and Daniel Tubbenhauer.  

This work was supported by Mitacs through the Mitacs Globalink program and by the Discovery Grant RGPIN-2018-03974 from the Natural Sciences and Engineering Research Council of Canada. This research was also supported by Perimeter Institute for Theoretical Physics. Research at Perimeter Institute is supported by the Government of Canada through the Department of Innovation, Science and Economic Development and by the Province of Ontario through the Ministry of Research and Innovation.

\section{Preliminaries}
\subsection{Hecke and cyclotomic $q$-Schur algebras}
\label{sec:Hecke}
Let $\Bbbk$ be a field of any characteristic.  Throughout, we will let $\mathcal{H}_n$ denote $\mathcal{H}_n^B(Q,q),$ the Hecke algebra of type B with generators $T_0, T_1, \dotsc , T_{n-1}$ subject to the relations
\begin{align*}
    (T_i-q_i)(T_i+1) & = 0 \\
    T_0T_1T_0T_1 &= T_1T_1T_0T_1 \\
    T_iT_{i+1}T_i &= T_{i+1}T_iT_{i+1} \text{ for } 1 \leq i \leq n-1 \\
    T_iT_j &= T_jT_i \text{ for } |i-j|>1
\end{align*}
where $q_i \in \Bbbk^{\times}$ are scalars.
We will work in the case where $q_0 = Q$ and $q_1 = q.$ Note that the braid relation also forces $q_i=q$ for $i = 2 , \dotsc , n-1.$

\begin{rmk}\label{rk:different-conventions}
Many different conventions for the roots of the minimal polynomials of the $T_i$ appear in the literature, but these are all equivalent up to algebra automorphism and parameter changes. We have chosen to follow the conventions of Dipper, James, and Mathas \cite{DJM}. 
\end{rmk}

We will also define $L_1 := T_0$ and inductively $L_{i+1} := q^{-1}T_iL_iT_i.$ These are the Jucys--Murphy elements as defined in \cite[\S 2]{DJM}. Recall that as a consequence of these defining relations, any symmetric polynomial in these Jucys--Murphy elements $L_i$ is central in $\Hh_n.$ In particular, note that \[T_i L_i L_{i+1}=L_iL_{i+1}T_i.\]

Let us now recall some more notation from \cite[\S 3]{DJM}: given a bicomposition $\lambda = (\lambda^{(1)}, \lambda^{(2)})$ of $n$,  where $\lambda^{(i)} = (\lambda_1^{(i)}, \dotsc , \lambda_{\ell_i}^{(i)})$ is a composition, let $\Sym_\lambda$ be the corresponding Young subgroup of the symmetric group $\Sym_n$. That is, \[\Sym_\lambda := \Sym_{\lambda_1^{(1)}} \times \dots \times \Sym_{\lambda_{\ell_1}^{(1)}} \times \Sym_{\lambda_1^{(2)}} \times \dots \Sym_{\lambda_{\ell_2}^{(2)}}.\] 

We define
\begin{equation}
    u_\lambda := \prod_{i = 1}^{|\lambda^{(1)}|} (L_i + 1)\qquad\qquad 
x_\lambda := \sum_{\omega \in \Sym_\lambda}T_\omega\qquad \qquad
m_\lambda := u_\lambda x_\lambda.
\end{equation}
Here, we depart slightly from the notation of \cite{DJM}; the element $u_\lambda$ would be denoted there by $u_{(0,|\lambda^{(1)}|)}^+$.   

Let $\mathcal{H}_n^A \subset \Hh_n$ be the copy of the type A Hecke algebra generated by $T_1, \dotsc, T_{n-1}.$ Note that $x_\lambda \in \mathcal{H}_n^A$ whereas $u_\lambda \notin \mathcal{H}_n^A$ unless $\lambda^{(1)}=\emptyset$. 

Let $\Lnm$ denote the set of bicompositions of $n$ such that first component has $n$ parts and the second $r=\lfloor \nicefrac{m}{2}\rfloor$ parts, and let $\Lpnm$ denote the set of bipartitions of $n$ such that the second component has at most $r$ parts.   Of course, if $m$ is large, then $\Lpnm=\Pi_n$ is the set of all bipartitions of $n$.  
\begin{lem}
  If $\mu \in \Lnm$ and  $\nu$ is an arbitrary bipartition of $n$ such that $\nu \vartriangleright \mu$, then   $\nu\in \Lpnm$.
\end{lem}
\begin{proof}
    Since $\nu$ has $n$ boxes, it can have at most $n$ non-zero parts, and so certainly the same is true of $\nu^{(1)}$.  Since $\nu \vartriangleright \mu$, we must have $|\nu^{(1)}|+\sum_{i=1}^m \nu_{i}^{(2)}\geq n$, so $\nu^{(2)}$ can have at most $m$ parts.
\end{proof}  These are exactly the conditions on a set of bicompositions  required at the start of \cite[\S 6]{LNX}.   
Let \[M^\lambda := m_\lambda \Hh_n \qquad\qquad M^{\Lnm}:=\bigoplus_{\lambda \in \Lnm} M^\lambda.\]

\begin{defn}We define the {\bf Dipper--James--Mathas Type B $q$-Schur algebra} of rank $n$ to be:
\begin{align*}
\Sd^n(\Lnm)&:= \Endh \left(M^{\Lnm} \right)
\cong \bigoplus_{\mu, \nu \in \Lnm}\operatorname{Hom}_{\Hh_n}(M^\mu , M^\nu).
\end{align*}
This is a special case of the cyclotomic $q$-Schur algebra \cite[Def 6.1]{DJM}.
\end{defn}

Fix bicompositions $\mu, \nu$ of $n$.  
\begin{lem}[\mbox{\cite[Corollary 5.17]{DJM}}]
    The map $\varphi \mapsto \varphi(m_{\mu})$ induces isomorphisms
\begin{equation*}
    \homh (M^\mu, M^\nu) \cong M^{\mu *} \cap M^\nu \qquad\qquad  \Sd^n(\Lnm ) \cong \bigoplus_{\mu, \nu \in \Lnm} M^{\mu *} \cap M^\nu.
\end{equation*}
That is, there is a (necessarily unique) morphism $\varphi \in \homh (M^\mu, M^\nu)$ such that $\varphi(m_{\mu})=m_{\nu}g$ if and only there is an $h_\varphi \in \Hh_n$ such that $m_{\nu}g = h_\varphi m_\mu.$ 
\end{lem}
This implies that to compute compositions of maps, one can use the fact that $h_{\phi \circ \psi} = h_\phi h_\psi.$

As usual, given multi-compositions $\lambda, \mu$, we say that a tableau of shape $\lambda$ is of {\bf type} $\mu$ when the entry $i_k$  occurs $\mu_i^{(k)}$ times.
Such a $\lambda$-tableau $T$ is {\bf semi-standard} if the entries across a row of a single partition are non-decreasing, those down a column are strictly increasing, and for any $k$ 
no entry of $T^{(k)}$ has the form $(i,l)$ with $l<k$.  Since we are only considering the case of bipartitions, this simply means that we cannot put an entry from the first alphabet into the second component.

Let $\mathcal{T}_0(\lambda,\mu)$ denote the set of semi-standard tableaux of shape $\lambda$ and type $\mu$, and  $\mathcal{T}(\lambda)=\mathcal{T}_0(\lambda, (\emptyset, (1^n)))$ denote the set of standard tableaux of shape $\lambda.$
\begin{thm}[\mbox{\cite[Theorem 6.12]{DJM}}]
The algebra $\Sd^n(\Lnm)$ has cellular basis \[\{ \varphi_{ST}: S \in \mathcal{T}_0(\lambda, \mu), T \in \mathcal{T}_0(\lambda, \nu) \text{ for some } \mu, \nu \in \Lnm, \lambda \in \Lpnm\}\]
\end{thm} 

Given a bicomposition of $n$, let $\varphi_\lambda \in \Sd^n(\Lnm)$ denote the identity map on the module $M^\lambda$. Note that this is an element of the cellular basis: $\varphi_\lambda = \varphi_{T^\lambda T^\lambda}$, where $T^\lambda$ is the unique semi-standard tableau of shape and type $\lambda$.

The cell modules of this cellular structure are the Weyl modules defined in \cite[Definition 6.13]{DJM}. We can write these as 
\[\Csd{\lambda}:=\Sd^n(\varphi_\lambda+\Lsd{\lambda})/\Lsd{\lambda}\subset \Sd^n/\Lsd{\lambda} \]
where the quotient is by the left ideal
\begin{equation*}
    \Lsd{\lambda} := \{ \varphi_{UV}: U \in \mathcal{T}_0(\alpha, \mu), V \in \mathcal{T}_0(\alpha, \nu) \text{ for some } \mu, \nu \in \Lnm, \alpha \in \Lpnm, \alpha \vartriangleright \lambda \}.
\end{equation*}
For a semi-standard $\lambda$-tableau $S$, let $\varphi_S= \varphi_{ST^\lambda} + \Lsd{\lambda}$ denote the coset of $\varphi_{ST^{\lambda}}$ in $\Csd{\lambda}$.
There is a unique inner product defined on these Weyl modules by the formula
\[ \varphi_{T^\lambda S} \varphi_{TT^\lambda} \equiv \langle \varphi_S, \varphi_T \rangle \varphi_\lambda \mod \Lsd{\lambda}.\]
This is always non-zero since $\langle\varphi_{T^\lambda},\varphi_{T^\lambda}\rangle=1 $ for all $\lambda$.  As we'll see below, this is a manifestation of the fact that $\Sd^n(\Lnm)$ is quasi-hereditary for all $m$ and all choices of parameters $q,Q\in \Bbbk^\times$.

\subsection{Type B $q$-Schur algebras}
\label{sec:type-B}

In \cite{LNX}, Lai, Nakano and Xiang consider a different endomorphism algebra, which we will denote $\Sl^n(m)$.

\begin{rmk}
The paper \cite{LNX} uses different conventions for the parameters of $\Hh_n^B(Q,q)$. As noted in Remark \ref{rk:different-conventions}, we have followed those of \cite{DJM}: the quadratic relations are $(T_0-Q)(T_0+1) = 0$ and $(T_i-q)(T_i+1)=0$ for $i \neq 0$. Where necessary, we have adjusted all theorems and formulas to match those of \cite{DJM}. In the paper \cite{LNX}, $d$ refers to the rank of the Hecke algebra. To follow the conventions of Dipper, James, and Mathas \cite {DJM}, we will call this $n$. Lai, Nakano, and Xiang use $n$ to refer to the dimension of the vector space that is being acted on by the Hecke algebra, which we have called $m$.  Note that a choice of isomorphism between the Hecke algebra in our conventions and those of \cite{LNX} requires a choice of square roots $\sqrt{Q},\sqrt{q}$ in our base field or in an extension.
\end{rmk}

As before, we let $r=\lfloor \nicefrac{m}{2}\rfloor$.  Let 
\begin{equation}\label{eq:m-prime}
    m'=m+2(n-r)=\begin{cases}
        2n+1 & m\text{ odd}\\
        2n & m\text{ even}
    \end{cases}
\end{equation}
Lai, Nakano, and Xiang define an index set $I(m) = \{-r, \dotsc -1, 0, 1, \dotsc , r\}$ if $m$ is odd and $\{-r, \dotsc -1, 1, \dotsc , r\}$ if $m$ is even.  Let $V=\Bbbk^{I(m)}$; note that this is an $m$-dimensional $\Bbbk$-vector space. The $n$-fold tensor product $V^{\otimes n}$ has a basis $\{ v_{(i_1,\dotsc ,i_n)}=v_{i_1}\otimes \cdots \otimes v_{i_n} : i_j \in I(m)\}$ with an $\Hh_n$-action that we describe below.

For $s_t \in W_{n}$ and any tuple $\vec{d} = (d_1, \dotsc , d_n) \in I(m)^n$ let
\begin{align*}
    (d_1 , \dotsc , d_n) \cdot s_t &= \begin{cases}
     (-d_1, d_2, \dotsc, d_n) & t = 0 \\
    (d_1 , \dotsc , d_{t-1}, d_{t+1}, d_t , d_{t+2}, \dotsc , d_n) & t \neq 0
    \end{cases}
\end{align*}
For $T_t \in \Hh_n$, the action is given by
\begin{align}\label{eq:Hecke-action}
    v_{\vec{d}} T_0 &= 
    \begin{cases}
    v_{\vec{d} \cdot s_0} & 0< d_1 \\
    Qv_{\vec{d} \cdot s_0} & 0 = d_1 \\
    v_{\vec{d} \cdot s_0} + (Q-1) v_{\vec{d}} & 0 > d_1
    \end{cases} &
    v_{\vec{d}} T_t &= 
    \begin{cases}
    v_{\vec{d} \cdot s_t} & d_t < d_{t+1} \\
    qv_{\vec{d} \cdot s_t} & d_t = d_{t+1} \\
    v_{\vec{d} \cdot s_t} + (q-1) v_{\vec{d}} & d_t > d_{t+1}
    \end{cases}
\end{align}
for all  $t \in \{ 1, \dotsc , n-1\}$.

\begin{defn}[\mbox{\cite[\S 2.3]{LNX}}]
$\Sl^n(m) := \Endh (V^{\otimes n}).$
\end{defn}
Since we will want to compare different values of $m$, assume that $m_1,m_2$ satisfy $m_1\leq m_2$ and $m_1$ is even or $m_2$ is odd. 
In this case, we have a natural inclusion $I(m_1)\hookrightarrow I(m_2)$; such an inclusion does not exist if $m_1$ is odd and $m_2$ is even, since in this case $0\in I(m_1)$ but $0\notin I(m_2)$.  From \eqref{eq:Hecke-action}, we see that the induced inclusion $V_1^{\otimes n}\hookrightarrow V_2^{\otimes n}$ is a map of $\Hh_n$-modules.  In fact, it is the inclusion of a summand; let $\emm{m_1}{m_2}$ be the natural projection that kills any tensor where one of the factors is not in $I(m_1)$.  Thus, one of our many appearances of corner algebras is that
\[\Sl^n(m_1)=\emm{m_1}{m_2}\Sl^n(m_2)\emm{m_1}{m_2}.\]
Note, this shows that if $m_1\leq m_2$ and $m_1$ is even or $m_2$ is odd, any \LNX{m_1} bipartition is also \LNX{m_2}.  

Consider the set \[\Anm = \{ \alpha = (\alpha_i)_{i \in I(m)} : \alpha_0 \in 2 \mathbb{Z} + 1, \alpha_{-i} = \alpha_i, \sum_i \alpha_i = 2n+1\}.\]Note that if $m$ is even, then $0\notin I(m)$, so the condition on $\alpha_0$ is vacuous.
\begin{defn}
For $\alpha \in \Anm$, we define the bicomposition $\lambda(\alpha) := (( \lfloor \frac{\alpha_0}{2} \rfloor ), (\alpha_1, \dotsc \alpha_n )).$  
\end{defn}
Let 
\begin{equation*}
   \LnmB:= \{ \lambda (\alpha) : \alpha \in \Anm \}\subseteq \Lnm.
\end{equation*}
Note that if $m$ is odd, then $\LnmB$ is the set of all bicompositions where the first component has one part and the second has $r$ parts and if $m$ is even, then it is all bicompositions where the first component is trivial and the second has $r$ parts.

Given $\alpha\in \Anm$ we define the set of transpositions
\[G_\alpha := \{ s_0, \dotsc , s_{n-1} \} - \{ s_{\lfloor \frac{\alpha_0}{2} \rfloor}, s_{\lfloor \frac{\alpha_0}{2} \rfloor + \alpha_1}, \dotsc , s_{\lfloor \frac{\alpha_0}{2} \rfloor + \alpha_1 + \dots + \alpha_r}\}.\]
The corresponding Weyl group $\bweyl{\alpha}= \langle G_\alpha \rangle$ is the subgroup generated by these transpositions. Note that this is the product of the type B Weyl group generated by $\{s_0, \dotsc s_{\lfloor \frac{\alpha_0}{2} \rfloor -1}\}$ with the type A Weyl group $\Sym_{\mu}$ for $\mu = (\alpha_1, \dotsc , \alpha_r)$. We will denote by $\Hh_{\alpha}$ the subalgebra of $\Hh_n$ generated by $\{ T_i : s_i \in G_{\alpha}\}$.

Consider the element
\[ x_\alpha^B := \sum_{\omega \in \bweyl{\alpha}} T_\omega = \sum_{T_\omega \in \Hh_\alpha} T_\omega. \]

By \cite[(2.5)]{LNX} we have an isomorphism
\[\Sl^n (m) \cong \bigoplus_{\alpha, \beta \in \Anm } \homh (x_{\alpha}^B \Hh_n, x_{\beta}^B \Hh_n).\]

An important observation is that the set $G_{\al}$ and thus the module $ x_{\alpha}^B \Hh_n$ is unchanged if we swap a part of $\al$ with one which is 0, or add additional 0's to the partition.  This shows that:
\begin{lem}\label{lem:m-prime-Morita}
    If $m\geq 2n$, then the idempotent $\emm{m'}{m}$ induces a Morita equivalence of $\Sl^n(m)$ to $\Sl^n(m')$.  
\end{lem}
\begin{proof}
    The composition $\alpha$ can only have $n$ non-zero parts, so if we remove all parts which are 0, we will arrive at $\alpha'\in \LnmB(m')$ such that $G_{\alpha}=G_{\alpha'}$, and so $x_{\alpha}^B \Hh_n\cong x_{\alpha'}^B \Hh_n$.  This shows that the image of $\emm{m'}{m}$ contains a copy of every indecomposable summand of $V^{\otimes n}$; the desired Morita equivalence follows immediately.  
\end{proof}

\begin{lem}\label{lem:unique-line}
For a fixed $\alpha\in \Anm$, there is a unique line $\ell \subset \Hh_\alpha$ such that $(T_i-q_i)\cdot \ell = 0$ for $i$ such that $s_i \in G_\alpha$, and it is spanned by $x_\alpha^B$.
\end{lem}
\begin{proof}

Consider an element $a=\sum_\omega a_\omega T_\omega \in \Hh_\alpha$ Using the decomposition of each group element $\omega$ into cosets of $s_i \in G_\alpha$ we obtain:

\[ a= \sum_{\omega \in \bweyl{\alpha}} (a_{s_i \omega} T_iT_\omega + a_\omega T_\omega)\]
so that
\begin{align*}
    T_ia &= \sum_{\omega \in \bweyl{\alpha}} (a_{s_i \omega}T_i^2T_\omega + a_\omega T_i T_\omega) \\
    &= \sum_{\omega \in \bweyl{\alpha}} (a_{s_i \omega}(q_i-1)T_i T_\omega + q_ia_{s_i\omega} T_\omega + a_\omega T_iT_\omega) \\
    &= \sum_{\omega \in \bweyl{\alpha}} ((a_{s_i \omega}(q_i-1)+a_\omega)T_iT_\omega + q_ia_{s_i\omega}T_\omega).
\end{align*}
Therefore, if we are to have $T_ia = q_ia$ we must have $a_{s_i \omega} = a_\omega$ for all $s_i \in \bweyl{\alpha}$.
\end{proof}    
\begin{lem}\label{lem:module-same}
Given $\alpha\in \Anm$, we have equality of the modules
\[ x_\alpha^B \Hh_n = m_{\lambda(\alpha)} \Hh_n\]
\end{lem}
\begin{proof}
    We can write $m_{\lambda}=x_{\lambda}u_{\lambda}=u_{\lambda}x_{\lambda}$ since $u_{\lambda}$ is a polynomial in the $L_i$'s which is symmetric under $\Sym_{\lambda}$.  Note also that if $\lambda^{(1)}\neq \emptyset$, then $u_{\lambda}$ is divisible by $T_0+1$, and so $(T_0-Q)u_{\lambda}=0$. 
    On the other hand, for any composition $\lambda$, we have that $x_{\lambda}=x_{\alpha'}^B$ for  \[\alpha'_i=\begin{cases}
    1 & i=0\\
    \lambda_i& i>0.
\end{cases}\] 
Applying  Lemma \ref{lem:unique-line} in this case, we find that if $s_i\in \Sym_{\lambda}$, then $(T_i-q)x_{\lambda}=0$.  
    Now, fix $\alpha$, and consider $m_{\lambda(\alpha)}$.  If $\alpha_0>1$, we have $s_0\in G_{\alpha}$.   Since $\lambda^{(1)}=(\lfloor \frac{\alpha_0}{2} \rfloor )$, we have that
    \[(T_0-Q)m_{\lambda(\alpha)}=(T_0-Q)u_{\lambda(\alpha)}x_{\lambda(\alpha)}=0.\]
On the other hand, if $s_i\in G_{\alpha}$ for $i>0$, then $s_i\in \Sym_{\lambda(\alpha)}$ and so 
\[(T_i-q)m_{\lambda(\alpha)}=(T_i-q)x_{\lambda(\alpha)}u_{\lambda(\alpha)}=0.\qedhere \]
\end{proof}
Thus, each direct summand of $V^{\otimes n}$ appears in $M^{\Lnm}$. However, not all summands appear, since $\LnmB$ is a proper subset of $\Lnm$ if $n>1$.  Let 
\begin{equation}\label{e-def}
    \e= \sum_{\lambda \in \LnmB} \varphi_\lambda.
\end{equation}  
Lemma \ref{lem:module-same} shows that as modules over the Hecke algebra:
\begin{equation}\label{tensormod}
    \e M^{\Lnm}=V^{\otimes n}.
\end{equation}
Thus, we have that:
\begin{thm}\label{idemp}
There is an isomorphism $\Sl^n(m) \cong \e \Sd^n(\Lnm) \e.$
\end{thm}
\begin{proof}   We note that $\varphi_\lambda$ is zero on all $M^\mu$ for $\mu \neq \lambda$. Thus, \begin{equation}\label{eq:compare-tensors}
       \e M^{\Lnm}\cong \bigoplus_{\lambda \in \LnmB} M^\lambda.
   \end{equation}
   Using the bijection of $\Anm$  with $\LnmB\subset \Lnm$ and the isomorphism of Lemma \ref{lem:module-same}, we can conclude $\e M^{\Lnm}=V^{\otimes n}$.  Thus, 
   \[ \e \Sd^n(\Lnm) \e =\Endh(\e M^{\Lnm})=\Endh(V^{\otimes n})=\Sl^n(m).\qedhere \]
\end{proof}
In particular, $\Sl^n(m)$ and $\Sd^n(\Lnm)$ are Morita equivalent if and only if no simple $\Sd^n(\Lnm)$ modules are killed by $\e$.

\subsection{Cellularity} To understand the consequences of this result, let us consider a general cellular algebra $A$ with cell data $(\Lambda,M,C,i)$,  cell modules $\cell{\lambda}$, and a general idempotent $e\in A$ such that $i(e)=e$.  Of course, we can consider the corner algebra or idempotent truncation $eAe$ for this idempotent.  

The example of matrix algebras shows that we can choose a cellular basis $C_{ST}$ where the number of non-zero vectors such that $eC_{ST}e\neq 0$ is larger than the dimension of $eAe$, meaning that these don't form a basis.  However, this defect is readily corrected by choosing a different cellular basis with the same associated cell modules and ideals.  
\begin{lem}[\mbox{\cite[Prop 4.3]{Konig-Xi-structure}}]\label{lem:corner-cell}    
    Under the assumptions above, the algebra $A$ has a cell datum $(\Lambda, M,C',i)$ with the same cell modules and ideals, such that $eAe$ is a cellular algebra with cell modules $e\cell{\lambda}$ and cell datum $(\Lambda^{(e)}, M^{(e)},C',i)$, where 
    \[ \Lambda^{(e)}=\{\la\in \Lambda \mid e\cell{\lambda}\neq 0\},\]
    \[M ^{(e)}(\lambda)=\{S\in M(\lambda) \mid eC'_S\neq 0\}\]
\end{lem}

In fact, since the idempotent $\e$ defined in \cref{e-def} acts by the identity or zero on each summand $M^{\nu}$, we do not need to choose a new basis:
\begin{cor}
The algebra $\Sl^n(m)$ is a cellular algebra with cellular basis
\[\{\e \varphi_{ST}\e: S \in \mathcal{T}_0(\lambda, \mu), T \in \mathcal{T}_0(\lambda, \nu) \text{ for some } \mu, \nu \in \LnmB, \lambda \in \Lpnm\}\]
and cell modules $\Csl{\lambda}=\e \Csd{\lambda}$.
\end{cor}   
Throughout, we will abuse notation and write $\varphi_{ST}$ for $\e \varphi_{ST}\e$ when $\mu, \nu \in \LnmB$.

\subsection{LNX bipartitions}
Whenever $A$ is a finite-dimensional $\K$-algebra, let $\Simp(A)$ denote the set of simple $A$-modules up to isomorphism.  Consider an idempotent $e\in A$.  We leave to the reader the proof of the following standard lemma:
\begin{lem}\hfill
\begin{enumerate}
    \item For each simple $eAe$ module $L$, the module $A\otimes_{eAe}L$ has a unique simple quotient $L^+=\operatorname{hd}(A\otimes_{eAe}S)$ such that $eL^+=L$.
    \item The map $L\mapsto L^+$ is an injective map $\Simp(eAe)\to \Simp(A)$ whose image is the simples $L'$ such that $eL'\neq 0$.
\end{enumerate}    
\end{lem}
All our different idempotent truncations (or the Morita equivalence of \cref{lem:m-prime-Morita} if $m\geq m'$) give us maps 
\[\begin{tikzcd}
	& {\Simp(\Sl^n(m))} & {\Simp(\Sd^n(m))} \\
	{\Simp(\Hh_n^B)} & {\Simp(\Sl^n(m'))} & {\Simp(\Sd^n(m'))}
	\arrow[hook, from=2-1, to=2-2]
	\arrow[hook, from=2-2, to=2-3]
	\arrow[hook, from=1-2, to=2-2]
	\arrow[hook, from=1-3, to=2-3]
	\arrow[hook, from=1-2, to=1-3]
\end{tikzcd}\]
The final set $\Simp(\Sd^n(m'))$ is identified with the set $\Pi_n$ of all bipartitions of $n$, so all the other algebras have sets of simples that are identified with  subsets of bipartitions.  A celebrated theorem of Ariki describes the image of $\Simp(\Hh_n)$: It is precisely the {\bf Kleshchev bipartitions}.  Analogously, we define:
\begin{defn}
    We call a bipartition \textbf{\LNX{m}} if $\Ssl{\lambda}=\e\Ssd{\lambda}\neq 0$.  We'll write \LNX{o} to mean \LNX{m} for $m$ large odd, and \LNX{e} for $m$ large even. 
\end{defn}

The main question we'll address in this paper is:
\begin{question}
    Which bipartitions are \LNX{m} for a given pair of parameters $(Q,q)$?
\end{question}
Obviously, all Kleshchev bipartitions are \LNX{m} for $m$ large.  However, we can construct a number of others (which in some cases are not Kleshchev).
\begin{lem}\label{lem:LB-LNX}
    Every bipartition in $\LnmB$ is \LNX{m} for all parameters $(Q,q)$.
\end{lem}
\begin{proof}
    If $\lambda \in \LnmB,$ then $\lambda = (\lambda^{(1)}, \lambda^{(2)}) = ((\lambda_1^{(1)}), \lambda^{(2)})$, i.e. the first partition has only one part. Let  $T^\lambda$ be the unique tableau in $\mathcal{T}_0(\lambda, \lambda).$ As noted above, $\varphi_{T^\lambda T^\lambda}$ is the identity map on $M^\lambda,$ and so
    \[ \langle \varphi_{T^\lambda}, \varphi_{T^\lambda}\rangle = 1\]
regardless of parameters. Therefore $\Ssl{\lambda} \neq 0$ and so $\lambda$ is \LNX{m}.
\end{proof}

\section{Quasi-hereditarity}
\subsection{Quasi-hereditarity of cellular algebras}

We will consider the question of whether $\Sl^n(m)$ is quasi-hereditary or equivalently, whether its category of modules is a highest weight category.  Refer to \cite[\S 3]{CPS} for the definitions of these notions.

Given a finite-dimensional cellular algebra $A$, let $n_s(A)$ be the number of isomorphism classes of simple $A$-modules, and let $n_c(A)$ be the number of cells.  By a theorem of Graham and Lehrer \cite[Th. 3.4]{grahamCellularAlgebras1996}, the number of simples
$n_s(A)$ is equal to the number of cell modules with non-zero bilinear form and, in particular, $n_s(A)\leq n_c(A)$.     The quasi-hereditarity of a cellular algebra is determined by the relationship between these numbers:
\begin{thm}[\mbox{\cite[Th. 3.1]{KX}}]\label{th:KX}
    The cellular algebra $A$ is quasi-hereditary if and only if $n_c(A)=n_s(A)$.
\end{thm}
Assuming we begin with a quasi-hereditary cellular algebra $A$, and then consider a corner algebra $eAe$, how can we determine if $eAe$ is quasi-hereditary?
By \cref{lem:corner-cell}, we have that $n_c(eAe)\leq n_c(A)$, with equality if and only if $\Lambda=\Lambda^{(e)}$.  On the other hand,
the algebras $A$ and $eAe$ are Morita equivalent if and only if $n_s(A) = n_s(eAe)$.  Combining these observations, we find that:
\begin{cor}\label{cor:corner-algebra}
    If $n_c(eAe)=n_c(A)$, and $A$ is quasi-hereditary, then the following are equivalent:
    \begin{enumerate}
        \item The algebra $eAe$ is quasi-hereditary.
        \item We have an equality $n_s(eAe)=n_s(A)$.
        \item The algebras $A$ and $eAe$ are Morita equivalent.
        \item The bilinear form is non-zero on $e\cell{\lambda}$ for all cell modules $\cell{\la}$ of $A$.  
    \end{enumerate}
\end{cor}
Let us apply this with $A=\Sd^n(\Lnm)$ and the idempotent $\e$ defined in \eqref{e-def}, resulting in the corner algebra $\e \Sd^n(\Lnm) \e =\Sl^n(m)$.  As noted in Section \ref{sec:Hecke}, the bilinear form on the Weyl modules $\Csd{\lambda}$ is always non-zero; one easy way to see this is that $\varphi_{\lambda}$ is a cellular basis vector and an idempotent.  Thus,  Theorem \ref{th:KX} confirms that $\Sd^n(\Lnm)$ is quasi-hereditary (see also \cite[Cor. 6.18] {DJM}). 

We consider the cell modules $\Csl{\lambda}=\e \Csd{\lambda}$.  For $m=2r+1$, this space has a basis indexed by semi-standard tableaux of shape $\lambda$ whose fillings lie in $\LnmB$.  Put another way, these are the semi-standard tableaux whose entries are from the set $\{1_1,1_2,2_2,\dots, r_2\}$.  For simplicity, we will replace $1_1$ with $0$, and drop the subscript $2$ from $k_2$, so we consider bitableaux with entries in $\{0,1,2,\dots, r\}$, where we only use $0$'s in the first component of the bipartition.  If $m=2r$, we don't use $0$ (i.e. $1_1$), and so we obtain the usual notion of a semi-standard tableau on a bipartition with entries in $\{1,2,\dots, r\}$.  

Thus, the cell module $\Csl{\lambda}$ is only non-zero if and only if $\lambda$ has at most $\lceil \nicefrac{m}{2}\rceil$ parts, at most $\lfloor\nicefrac{m}{2}\rfloor $ of them in $\la^{(2)}$.  In particular, $\Csl{((1^n),\emptyset)}=\e\Csd{((1^n),\emptyset)}\neq 0$ if and only if $m$ is large; note that we always have  $((1^n),\emptyset)\in \Lpnm$ so if $m<2n$, then $\Sl^n(m)$ will have strictly fewer cells than $\Sd^n(\Lnm)$.  

Now, we focus on the case where $m$ is large. 
In this case, we can consider the idempotent $\eh{n}=\varphi_{(\emptyset,1^n)}$; since $\eh{n}V^{\otimes n}$ is a free module of rank 1 over the Hecke algebra, we have  $\eh{n}\Sd^n(\Lnm)\eh{n}=\mathcal{H}_n$. The basis vectors of $\eh{n}\Csl{\lambda}$ correspond to standard tableaux with the entries $\{1,\dots, n\}$.  The module $\Ch{\lambda}=\eh{n}\Csl{\lambda}=\eh{n}\e\Csd{\lambda}$ is non-zero for all bipartitions $\lambda$.  

This shows that whenever $m$ is large, we have 
\begin{math}
    n_c(\Sd^n(\Lnm))=n_c(\Sl^n(m))=n_c(\mathcal{H}_n).
\end{math}
The question of quasi-hereditarity is resolved for $\eh{n}$, though Corollary \ref{cor:corner-algebra} is a different interpretation of this result: the algebra $\mathcal{H}_n$ is quasi-hereditary if and only if $\mathcal{H}_n$ is semi-simple (see \cite[Thm 3.13]{mathas02}).

By the main theorem of \cite{Ariki}, $\Hh_n$ is semi-simple and quasi-hereditary if and only if 
\[P_n(Q,q)=\prod_{m=2}^{n-1}\frac{q^{m}-1}{q-1}\cdot \prod_{i=1-n}^{n-1}(Q+q^i)\neq 0.\]
These are precisely the parameters where every bipartition is Kleshchev.

We'll see below that the question of quasi-hereditarity for $\Sl^n(m)$ is more subtle, but we can still apply Corollary \ref{cor:corner-algebra} in this case to conclude that the algebra $\Sl^n(m)$ is quasi-hereditary if and only if the restriction of the bilinear form remains non-zero on $\Csl{\lambda}$ for all $\lambda\in \Lambda_n^B(m)$.  That is:
\begin{cor}\label{cor:bipartitions}
    For $m$ large, the algebra $\Sl^n(m)$ is quasi-hereditary if and only if every bipartition of $n$ is \LNX{m}.    
\end{cor}

\begin{rmk}
    Here we see immediately why $m<2n$ will prove to be more difficult.  In this case, we have a strict inequality $n_c(\Sd^n(\Lnm))>n_c(\Sl^n(m))$ since we have already noted that $\e\Csd{((1^n),\emptyset)}=0$. Thus, there can be bipartitions which are not \LNX{m}, but which do not cause a problem for quasi-hereditarity since their cell modules are trivial as well.
\end{rmk}


Consider the polynomial $f_n(Q,q) = \prod_{i=1-n}^{n-1}(Q+q^{i})$.  Recall again that we have shifted conventions from \cite{LNX}, so we have rewritten this polynomial in the conventions of \cite{DJM}. This is recognizable as a factor of $P_n(Q,q)$ above: if $q$ is not a root of unity, we have $f_n(Q,q)=0$ if and only if $P_n(Q,q)=0$. If this polynomial doesn't vanish, quasi-hereditarity is resolved:
\begin{lem}[\mbox{\cite[Cor. 6.1.1]{LNX}}]\label{lem:f-neq}
	If $f_n(Q,q)\neq 0$, then $\Sl^n(m)$ is quasi-hereditary for all $m$ and all bipartitions of $n$ are  \LNX{m}.  
\end{lem}
In \cite[Conj. 6.1.3]{LNX}, it is conjectured that the converse to this result holds.  As stated in \cref{th:A}, this conjecture needs appropriate modification before it holds if $m$ is odd, although it holds unchanged if $m$ is even.  Our main goal in the remainder of the paper will be to establish these results.
\subsection{Counterexample}
\label{sec:counterexample}

Let us first show that the converse to \cref{lem:f-neq} is not true in a low rank case: in the case where $n=2$, if $Q \in  \{-q^{-1},1\}$, then $f_n(Q,q) = \prod_{i=-1}^1(Q+q^i)=0$, but the algebra $\Sl^n(m)$ is still quasi-hereditary for $m$ large odd, for example, $m=2n+1=5$.  

We consider the Hecke algebra generated by $T_0, T_1$ subject to the relations $T_1^2 = (q-1)T_1+q,$ $T_0^2 = (Q-1)T_0+Q,$ and $T_1T_0T_1T_0 = T_0T_1T_0T_1.$ 
The bipartitions of $2$ are:
\begin{align*}
    \lambda &= ((2), \emptyset) = \left( \ydiagram{2}, \emptyset \right) &  \lambda' &= (\emptyset, (2)) = \left( \emptyset , \ydiagram{2} \right) \\
    \mu &= ((1^2), \emptyset) = \left( \ydiagram{1,1}, \emptyset \right) &   \mu' &= (\emptyset, (1^2)) = \left( \emptyset, \ydiagram{1, 1} \right) \\
    \nu &= ((1), (1)) = \left( \ydiagram{1}, \ydiagram{1} \right) 
\end{align*}
Note that $S=\{\lambda,\nu,\lambda',\mu'\}\subset \LnmB$ and thus all of these bipartitions are \LNX{m} for all parameters all $m\geq 4$, but $\mu\notin \LnmB$, so we must investigate further to see if it is \LNX{m}.  
By Corollary \ref{cor:corner-algebra}, the algebra $\Sl^2(m)$ will be quasi-hereditary if and only if the bilinear form is non-zero on $e_2\Csd{\mu}$. So, let us compute this bilinear form, assuming that $m$ is large odd.  

From the definition, we have:
\begin{align*}
    m_\lambda &= (L_1 +1)(L_2+1)(T_1+1) & m_{\lambda'} &= T_1+1 \\
    m_\mu &= (L_1 +1)(L_2+1) & m_{\mu'}&=1\\
    m_\nu &= L_1 +1
\end{align*}

For our counterexample, we investigate the cellular basis vectors corresponding to the bipartition 
$\mu := ((1)^2, \emptyset)$.
The tableaux of shape $\mu$ using the above described alphabet are
\begin{align*}
    A &= \left( \begin{ytableau}
0 \\ 1 
\end{ytableau} , \emptyset \right) &
    B &= \left( \begin{ytableau}
1 \\ 2 
\end{ytableau} , \emptyset \right) &
C &= \left( \begin{ytableau}
0 \\ 2 
\end{ytableau} , \emptyset \right)
\end{align*}

As $C$ is equivalent to $A$ via the isomorphism $M^{((1), (1,0))} \cong M^{((1), (0,1))}$, we need only consider $A$ and $B$.  Thus, we just need to calculate the bilinear forms $\langle \varphi_{A}, \varphi_{A} \rangle$ and $\langle \varphi_{B}, \varphi_{B} \rangle$, and quasi-hereditarity will fail if and only if these are both 0.  
\begin{lem}
    We have equalities:
    \[\langle \varphi_A, \varphi_A \rangle = q^{-1}Q+1\qquad\qquad  \langle \varphi_B, \varphi_B \rangle = (Q+1)(q^{-1}Q+1).\]
\end{lem}
\begin{proof}
    We begin with the endomorphism
\[\varphi_{AA} \in End_{\Hh_n}(m_\nu)\qquad \qquad\varphi_{AA}(m_\nu) = m_{\mu}=(L_2+1)m_\nu\] 
By the argument given in the proof of \cite[Prop. 3.7]{jf}
\begin{align*}
    \varphi_{AA}^2(m_\nu) &= (L_2+1)m_\mu \\
    &\equiv (\res_\mu (2) + 1)m_\mu \mod \Lh{\mu}\\
    &\equiv (q^{-1}Q+1) \varphi_{AA}(m_\nu) \mod \Lh{\mu}
\end{align*}
which implies that $\langle \varphi_A, \varphi_A \rangle = q^{-1}Q+1$.
Now consider
\[\varphi_{BB} \in \Endh(M^{\mu'})\qquad \qquad\varphi_{BB}(m_{\mu'}) = m_\mu.\] 
Following the same argument as above,
\begin{align*}
    \varphi_{BB}^2(m_{\mu'}) &= m_\mu^2 \\
   &= (L_1+1)(L_2+1) m_\mu \\
   & \equiv (Q+1)(q^{-1}Q+1) \varphi_{BB}(m_{\mu'})
\end{align*}
implying that $\langle \varphi_B, \varphi_B \rangle = (Q+1)(q^{-1}Q+1)$.
\end{proof}
Finally we need the calculation
\[\langle \varphi_D,\varphi_D\rangle =(Q+1)(Qq+1) \qquad \qquad D = \left( \begin{ytableau}
1 & 1 
\end{ytableau} , \emptyset \right) .\]
\begin{prop}
    The algebra $\Sl^2(m)$ is quasi-hereditary if and only if there is a $\checkmark$ in the corresponding entry of the table below:
\begin{center}
    \begin{tabular}{|c|c|c|c|c|}\hline
    & $Q=-q$ & $Q=-1$ & $Q=-q^{-1}$ & $Q\notin \{-q,-1,-q^{-1}\}$
    \\\hline
        $m\geq 5$ odd & \texttimes &  \checkmark & \checkmark & \checkmark\\\hline
        $m\geq 4$ even & \texttimes &\texttimes &\texttimes & \checkmark\\\hline
                $m= 3$ & \texttimes &\checkmark &\checkmark & \checkmark\\\hline
                $m= 2$ & \checkmark &\texttimes &\texttimes & \checkmark\\\hline
                $m= 1$ & \checkmark &\checkmark &\checkmark & \checkmark\\\hline
    \end{tabular}
\end{center}
If $q=-1$ and $Q=1$, then we must have checkmarks in both columns whose equations hold, that is, the algebra $\Sl^2(m)$ is quasi-hereditary if and only if $m=1$.
\end{prop}
Thus, there are many choices of $m$ and $(Q,q)$ where $f_n(Q,q) = (Q+q^{-1})(Q+1)(Q+q)=0$, but $\Sl^2(m)$ is quasi-hereditary, showing that a revision of 
\cite[Conj. 6.1.3]{LNX} is needed.
\begin{proof}
First note that the case $Q\notin \{-q,-1,-q^{-1}\}$ is covered by Lemma \ref{lem:f-neq}.

    {\bf The case of $m\geq 3$ and odd:}  As calculated above $\Csl{\mu}\neq 0$, but the bilinear form on this module vanishes if and only if $Q=-q$.
    
    {\bf The case of $m\geq 4$ and even:} In this case, $\Csl{\mu}$ is generated by $\varphi_B$, and so this bipartition is not \LNX{m} if $Q\in \{-1,-q\}$.  
On the other hand, the module  $\Csl{\la}$ is generated by $\varphi_D$, and so the bilinear form vanishes on this module  if and only if $Q\in \{-q^{-1},-1\}$ (this is a special case of \cref{formula-even}).  Thus, $\Sl^2(m)$ is quasi-hereditary if and only if $Q\notin \{-q,-1,-q^{-1}\}$.

{\bf The case of $m=2$:} 
If $m=2$, then $\Csl{\mu}=0$, so $\Sl^2(m)$  is quasi-hereditary if and only if $\la$ is \LNX{2}, i.e. if $Q\notin \{-q^{-1},-1\}$.

    {\bf The case of  $m=1$:}  In this case, $\Sl^2(m)\cong \K$, so this algebra is semi-simple, and thus quasi-hereditary for all $q,Q$.
\end{proof}
One can check that if $q\neq -1$, then in all the ``bad'' cases for $m\geq 4$, the algebra $\Sl^2(m)$ is Morita equivalent to the Hecke algebra $\Hh_2$.  If $q=-1$, then $\Sl^2(m)$ is not Morita equivalent to $\Hh_2$ for any value of $Q$ since $\mu'$ is \LNX{m} for $m\geq 4$ but not Kleshchev.  


Taking $Q=-q^{-1}$ or $Q=-1$ gives $f_n(Q,q) = (Q+q^{-1})(Q+1)(Q+q)=0$.  On the other hand,
the inner product $\langle \varphi_{A}, \varphi_{A} \rangle$ is nonzero so $\Sl^2(m)$ is quasi-hereditary for $m$ large odd, showing that a revision of 
\cite[Conj. 6.1.3]{LNX} is needed.

\section{Unstacking}
To generalize this example, we consider the case of a bipartition $\lambda$ of $n$ such that $\lambda^{(2)}=\emptyset$. We will drop the superscripts on the parts for simplicity:
\[ \lambda = ((\lambda_1, \dotsc , \lambda_\ell), \emptyset).\]

Let $\un(\lambda)$ denote the bipartition given by unstacking the diagram for $\lambda$ after the first row, and moving the second row and onward to the second partition, i.e.:
\[ \un(\lambda) = ((\lambda_1), (\lambda_2, \dotsc , \lambda_\ell)).\]
Note that
\[ \Sym_{\un(\lambda)} \cong \Sym_\lambda \cong \Sym_{\lambda_1} \times \Sym_{\lambda_2} \times \dots \times \Sym_{\lambda_\ell}\]
giving us that $x_\lambda = x_{\un(\lambda)}$ and therefore that
\[ m_\lambda = \left(\prod_{i=\lambda_1+1}^n(L_i+1)\right)m_{\un(\lambda)}.\]

Given any tableau $A$ of shape $\lambda$ with any filling, let $\un (A)$ denote the corresponding tableau of shape $\un(\lambda).$ This ``unstacking'' operation on Young diagrams has inverse given by ``stacking'' the two partitions: given a bicomposition $\rho = (\rho^{(1)}, \rho^{(2)})$, let $\st(\rho)$ denote the diagram with $\rho^{(1)}$ stacked on top of $\rho^{(2)}$ and $\emptyset$ in the second component:
\[ \st(\rho) = ( (\rho^{(1)}, \rho^{(2)}), \emptyset).\]

Note that stacking or unstacking a tableau does not change its type (i.e. filling). The only thing that may change is whether or not we maintain increasing entries down the columns.

Given any bipartition $\alpha$, bicomposition $\mu,$ and semi-standard tableau $S \in \mathcal{T}_0(\alpha, \mu)$ let

\[ \mathcal{A}(\alpha, S) = \{ s \in \std(\alpha) | \mu(s)=S\}.\]

\begin{lem}{\label{seteq}}
    Fix a semi-standard tableau $S \in \mathcal{T}_0(\lambda, \mu)$ of shape $\lambda$ and type $\mu$. We have equalities of sets of standard tableaux:
\[ \{ \un(s)| s \in \mathcal{A}(\lambda, S)\} = \mathcal{A}(\un(\lambda), \un(S))\]
\[ \mathcal{A}(\lambda, S) = \{ \st(s')|s'\in \mathcal{A}(\un(\lambda), \un(S)\}\]
Stacking and unstacking induce inverse bijections between these sets.
\end{lem}
\begin{proof}
    Given $s \in \std(\lambda)$ such that $\mu(s) = S$, we have $\un(s) \in \std(\un(\lambda))$ and $\mu(\un(S)) = \un(S).$ Consider $s' \in \std(\un(\lambda))$ such that $\mu(s') = \un(S)$. Since $\mu(s')$ sends an integer $a$ to $i_k$ whenever $a$ is in the $i^{th}$ row of the $k^{th}$ component of $t^\mu$, $\mu(\st(s'))=S.$ What remains to be verified is that $\st(s') \in \std(\lambda)$.

    Since entries are of weight $(\emptyset, (1^n))$ and increase along rows in $s',$ they must also increase along rows in $\st(s').$ We now consider the $i^{th}$ column of $\st(s').$ Let $a_r$ denote the entry in row $r$ of column $i$, and let $x_r$ denote the entry in the same position in $S$. Since $S$ is a semi-standard tableau, we must have $x_1 < x_2$. The map given by $\mu: a_r \mapsto x_r$ is weakly order preserving, so if we had $a_1 > a_2$ we would have had $x_1 \geq x_2$, a contradiction. Therefore $a_1<a_2$, and so $\st(s') \in \std(\lambda).$
\end{proof}
Given a standard $\un(\lambda)$ tableau $t$, recall from \cite{DJM} that $d(t) \in \Sym_n$ is the unique shortest permutation such that $t\cdot d(t) = t^{\un(\lambda)}$. For any $s \in \std(\lambda)$ and $s' \in \un(\lambda),$
\[ d(s) = d(\un(s)), \qquad d(s') = d(\st(s')).\]

\begin{lem}\label{lem:T-factor}
    Let $S$ and $\lambda$ be as above. Let $T \in \mathcal{T}_0(\lambda, \un(\lambda))$ be the semi-standard tableau with $0$'s in row $1$ and $r$'s in row $r+1$ for $1 \leq r \leq \ell-1:$
\[ T= \left( \begin{ytableau}
0 & 0 & 0 & \dots & 0 & \dots & 0 \\ 1 & 1 & 1 & \dots & 1 \\ \dots & & & \\ \teeny{\ell-1} & \dots & \teeny{\ell-1}
\end{ytableau} , \emptyset \right)\]
For any $j$, $1\leq j \leq \ell$ we have an equality:
\[ \varphi_{ST} = \varphi_{\un(S) \un(T)}\circ \varphi_{TT}\]
\end{lem}
\begin{proof}
    First, note that these maps can indeed be composed:
    \[ \varphi_{ST}, \varphi_{\un(S) \un(T)} \in \homh (M^{\un(\lambda)}, M^\mu), \qquad \varphi_{TT} \in \homh (M^{\un(\lambda)}, M^{\un(\lambda)})\]

    If $t \in \std(\un(\lambda))$ is such that $\un(\lambda)(t) = \un(T)$, then $t = t^{\un(\lambda)}$ and $d(t) = 1.$ Therefore,
    \[ m_{TT} = m_\lambda, \qquad m_{ST} = \sum_{s \in \std(\lambda), \mu(s) = S}T_{d(s)}^*m_\lambda, \qquad m_{\un(\lambda)}=\sum_{s' \in \std(\un(\lambda),\mu(s')= \un(S)}T_{d(s')}^*m_{\un(\lambda)}\]
Combining this with \ref{seteq}, we get
\begin{align*}\varphi_{TT}\left(m_{\un(\lambda)}\right)
    &= \left( \prod_{i=\lambda_1+1}^n(L_i+1) \right) m_{\un(\lambda)}
\end{align*}
so that
\begin{align*}
    \varphi_{\un(S)\un(T)}\circ \varphi_{TT}\left( m_{\un(\lambda)}\right) &= \varphi_{\un(S)\un(T)} \left( m_\lambda \right) \\
    &= \varphi_{\un(S)\un(T)}\left( m_{\un(\lambda)} \prod_{i=\lambda_1+1}^n(L_i+1) \right)\\
    &=m_{\un(S)\un(T)} \prod_{i=\lambda_1+1}^{n}(L_i+1)\\
    &= \left( \sum_{s' \in \mathcal{A}(\un(\lambda), \un(S))}T_{d(s')}^*m_{\un(\lambda)}\right) \left(  \prod_{i=\lambda_1+1}^{n}(L_i+1) \right) \\
    &= \sum_{s' \in \mathcal{A}(\un(\lambda), \un(S))}T_{d(s')}^*m_\lambda
\end{align*}
On the other hand, by \ref{seteq} and since $T_{d(s)}=T_{d(\un(s))},$
\begin{align*}
    \varphi_{ST}(m_{\un(\lambda)}) &= \sum_{s \in \mathcal{A}(\lambda, S)}T_{d(s)}^*m_\lambda \\
    &= \sum_{\un(s), s \in \mathcal{A}(\lambda, S)} T_{d(\un(s))}^*m_\lambda \\
    &= \sum_{s' \in \mathcal{A}(\un(\lambda), \un(S))}T_{d(s')}^*m_\lambda
\end{align*}
giving us the result.
\end{proof}
Since $T$ is the only tableau of its type and shape, we have that $\langle \varphi_T,\varphi_S\rangle=0$ for all $S\neq T$.  Thus, $\varphi_T$ lies in the radical of the cell module form if and only $\langle \varphi_T, \varphi_T\rangle = 0$.  Therefore:

\begin{cor}  For $m$ large odd, the vector $\varphi_T$ generates $\Csl{\lambda}$, and the inner product on this cell module will be zero if and only if $\langle \varphi_T, \varphi_T\rangle = 0.$
\end{cor}

Following the proof of \cite[Prop. 3.7]{jf}, we have
\[ m_\lambda L_i \equiv \res_{\lambda}(i)m_\lambda \mod \Lh{\lambda}.\]

So,
\begin{align*}
\varphi_{TT}^2(m_{\un(\lambda)}) &= \prod_{i=\lambda_1+1}^n(L_i+1)\varphi_{TT}(m_{\un(\lambda)}) \\
    & = \prod_{i=\lambda_1+1}^n(L_i+1) m_\lambda 
    \\
    &\equiv \prod_{i=\lambda_1+1}^n(\res_{\lambda}(i)+1)m_\lambda \mod \Lh{\lambda}\\
\end{align*}
which implies that
\[ \langle \varphi_T, \varphi_T \rangle = \prod_{i=\lambda_1+1}^n (\res_{\lambda}(i)+1).\]
Each row from $k=2$ onwards will contribute a factor of
\[ \prod_{b=1}^{\lambda_k}(Qq^{b-k}+1).\]
Taking the product over all of these rows gives:
\begin{equation}
    \langle \varphi_T, \varphi_T \rangle = \prod_{a=2}^{\ell} \left( \prod_{b=1}^{\lambda_a}(Qq^{b-a}+1) \right).
\end{equation}

\begin{cor}\label{formula}   The bipartition
    $\lambda = ((\lambda_1, \dotsc , \lambda_\ell), \emptyset)$ is \LNX{o} if and only if \[\prod_{a=2}^{\ell} \left( \prod_{b=1}^{\lambda_a}(Q+q^{a-b}) \right) \neq 0.\]  In particular:
    \begin{enumerate}
        \item    The bipartition $\lambda = ((1^n), \emptyset)$ is \LNX{o} if and only if $\prod_{i=2}^n(Q+q^{i-1}) \neq 0$.
        \item The bipartition $\lambda = \left( \left( \lceil \frac{n}{2}\rceil, \lfloor \frac{n}{2} \rfloor \right), \emptyset \right)$ is \LNX{o} if and only if $\prod_{2-\lfloor \frac{n}{2} \rfloor}^1(Q+q^i) \neq 0.$
    \end{enumerate}
\end{cor}
This result can be interpreted using the residue function 
\begin{equation}\label{content-def}
    \operatorname{res}(r,c,\ell)=\begin{cases}
    Qq^{c-r} & \ell=1\\
    -q^{c-r} & \ell=2.
\end{cases}
\end{equation} 
where $(r,c,\ell)$ denotes the box in row $r$ and column $c$ of the $\ell$-th component.

We can rephrase the result above as saying that $\lambda$ is \LNX{o} if and only if no box $(a,b,1)$ in the diagram of $\la^{(1)}$ with $a>1$ has residue $-1$.

\begin{lem}\label{klimits}
    If any bipartition of the form $\lambda = ((\lambda_1 , \dotsc \lambda_\ell ) , \emptyset)$ is not \LNX{o}, then $Q=-q^k$ for some 
    $k\in [2 - \lfloor \frac{n}{2} \rfloor,n-1]$.  
\end{lem}
\begin{proof}
    By \cref{formula}, we must have $Q=-q^{a-b}$, and there is a box $(a,b)$ with $a>1$ in the diagram of $\lambda$.  We find the limits above by trying to maximize or minimize $a-b$. 
    
    Of course, $a-b<a\leq n$, which gives the desired upper bound.  
On the other hand, we must have \[b-a\leq b-2\leq \la_2-2\leq \frac{\la_1+\la_2}{2}-2\leq \frac{n}{2}-2.\]  This implies that $b-a\leq \lfloor \frac{n}{2} \rfloor-2$, completing the proof.
\end{proof}


Note that if $q$ is itself a root of unity of order $e$, then $Q=-q^k$ can be reduced to $Q=-q^{\overline{k}}$ with $0 \leq \overline{k} <e$, so that these bounds are extremal regardless of the quantum characteristic.

\subsection{Case of $m$ even}
We will discuss it in less detail, but a similar calculation applies when $m$ is even.  In this case, the natural generator for the cell module $\Csl{ ((\lambda_1 , \dotsc \lambda_\ell ) , \emptyset)}$ is $U$, the unique semi-standard tableau of type $ ( \emptyset,(\lambda_1 , \dotsc \lambda_\ell ))$; we can construct this by shifting up all the entries of $T$ by 1.  \cref{lem:T-factor} generalizes readily to show that $\varphi_U$ generates $\Csl{ ((\lambda_1 , \dotsc \lambda_\ell ) , \emptyset)}$ in the case of $m$ large even, so \cref{formula} is modified to:
\begin{cor}\label{formula-even}  The bipartition
    $\lambda = ((\lambda_1, \dotsc , \lambda_\ell), \emptyset)$ is  \LNX{e} if and only if \[\prod_{a=1}^{\ell} \left( \prod_{b=1}^{\lambda_a}(Q+q^{a-b}) \right) \neq 0\] i.e. if and only if no box $(a,b,1)$ in the diagram of $\la^{(1)}$ has residue $-1$.  In particular:
    \begin{enumerate}
        \item    The bipartition $\lambda = ((1^n), \emptyset)$ is  \LNX{e} if and only if $\prod_{i=1}^n(Q+q^{i-1}) \neq 0$.
        \item The bipartition $\lambda = (  (n), \emptyset )$ is  \LNX{e} if and only if $\prod_{i=1}^{n}(Q+q^{-i+1}) \neq 0.$
    \end{enumerate}
\end{cor}

\section{Induction and Restriction Functors}
\subsection{Definition and basic properties}  Fix $n \in \mathbb{Z}_{>0}$ and $\lambda \in \Lambda_n$. 
In this section, we fix $m=m'$, so either $m=2n+1$ or $m=2n$; note, that we should interpret this to mean that $\Sl^{n\pm 1}=\Sl^{n\pm 1}(2(n\pm 1)+1)$ or $\Sl^{n\pm 1}=\Sl^{n\pm 1}(2(n\pm 1))$.  

Let $\hres{n+1}{n}: \mathcal{H}_{n+1}\operatorname{-mod} \rightarrow \mathcal{H}_n\operatorname{-mod}$ and $\hind{n}{n+1}: \Hh_{n} \rightarrow \Hh_{n+1}$ denote the induction and restriction functors on the categories of right modules over the Hecke algebra.

We also recall the induction and restriction functors of \cite{wada} for right $\Sd^n$ modules:
\[ \Ind{n}{n+1}: \Sd^n\operatorname{-mod} \rightarrow \Sd^{n+1}\operatorname{-mod}\]
\[ \Res{n+1}{n}: \Sd^n\operatorname{-mod} \rightarrow \Sd^{n-1}\operatorname{-mod}\]
given by 
\[ \Ind{n}{n+1} (N) =\Sd^{n+1} \xi \otimes_{\Sd^n}N \qquad \qquad  \Res{n+1}{n}(M) = \Hom_{\Sd^{n+1}}(\Sd^{n+1}, M).\]
where $\xi: \Sd^n \rightarrow \Sd^{n+1}$ is the injection induced by adding the lowest addable node to a bipartition of $n$.

Let $\Omega_n: \Sd^n\operatorname{-mod} \rightarrow \Hh_n\operatorname{-mod}$ be the Schur functor $ \Omega_n (M) = M\eh{n}$.

\begin{cor}[\mbox{\cite[Cor. 3.18]{wada}}]\label{cor:wada}
There is an isomorphism of functors
\[ \Omega_n \circ \Res{n+1}{n} \cong {}^{\mathcal{H}\!} \Res{n+1}{n} \circ \Omega_{n+1}\]
and
\[ \Omega_{n+1} \circ \Ind{n}{n+1} \cong {}^{\mathcal{H}\!}\Ind{n}{n+1} \circ \Omega_n\]
\end{cor}
Let $\oml{n}: \Sd^n\operatorname{-mod} \rightarrow \Sl^n\operatorname{-mod}$ be the truncation functor sending  $\oml{n} (M) = M\e$.
Given any associative ring $R,$ let $\F{k} : R\operatorname{-mod} \rightarrow R\operatorname{-mod}$ be the functor given by
\[ \F{k}(M) = M^{\oplus k} \]
Recall that by \eqref{eq:compare-tensors}, we have an isomorphism:
\[ e_n \Sd^n\eh{n}\cong V^{\otimes n}\]
where $V$ is a $m$-dimensional vector space. 
\begin{lem}
    For any $M \in \Sd^n\operatorname{-mod},$ 
    \[\Ind{n}{n+1}(M)e_{n+1} \cong (Me_n)^{\oplus \dim V}\]
\end{lem}
\begin{proof}
   First, assume that $P$ is a projective right $\Sd^n$ module.
 Then, we have
\begin{align*}
\Ind{n}{n+1}(P)e_{n+1} &\cong \Hom_{\Sd^{n+1}}(e_{n+1}\Sd^{n+1}, \Ind{n}{n+1}(P)).
\end{align*}
By the adjunction of \cite[Corollary 3.18(ii)]{wada}
\begin{align*}
    \Ind{n}{n+1}(P)e_{n+1} &\cong \Hom_{\Sd^n}(\Res{n+1}{n}(e_{n+1}S^{n+1}), P) \\
    &\cong \Hom_{\Hh^n}(\Omega_n(\hres{n+1}{n}(e_{n+1}S^{n+1})), \Omega_n(P))).
\end{align*}
since the Schur functor is fully faithful on projectives \cite[4.8]{wada}. Now, restricting the action of $\Hh_{n+1}$ to $\Hh_n$ gives us the isomorphism of $\Hh_n$ modules:
\begin{align*}
 e_{n+1}S^{n+1}\eh{n} &\cong V^{\otimes n+1}\\
 &\cong (V^{\otimes n}\otimes V)\\
 &\cong (e_n \Sd^n \eh{n})^{\oplus m}
\end{align*}
where $V$ is the vector space from the construction of $\Sl^n$ in \cite{LNX}, so that 
\begin{align*}
    \Hom_{\Hh_n}(\Omega_n(\hres{n+1}{n}(e_{n+1}S^{n+1}), \Omega_n(P))) &\cong \Hom_{\Hh_n}(\Omega_n((e_n\Sd^n)^{\oplus m}), \Omega_n(P)).
\end{align*}
Since $\Omega_n$ is fully faithful on projectives, this is isomorphic to
\begin{align*}
    \Hom_{\Sd^n}((e_n\Sd^n)^{\oplus m}, P)
    &\cong (P\e)^{\oplus m}
\end{align*}
giving us
\[ \Ind{n}{n+1}(P)e_{n+1} \cong (Pe_n)^{\oplus m}.\]

Now, let $M$ be an arbitrary right $\Sd^n$ module. We can write $M$ as a cokernel of a map between projectives $P_1, P_2 \in \Sd^n\operatorname{-proj}$. So, we have an exact sequence:
\[ P_2 \overset{f}{\rightarrow} P_1 \rightarrow M \rightarrow 0 \]
By \cite[Corollary 3.18(i)]{wada}, induction is exact, as is multiplying with any idempotent, so we have the exact sequence

\begin{equation}
    \Ind{n}{n+1}(P_2)e_{n+1} \rightarrow \Ind{n}{n+1}(P_1)e_{n+1}\rightarrow \Ind{n}{n+1}(M)e_{n+1} \rightarrow 0
\end{equation}
Applying the exact functors $\oml{n+1} \circ \Ind{n}{n+1}$ and $\F{k}\circ \Omega_n$ to this sequence, joining the resulting sequences by the isomorphisms on projectives, and then applying the Five Lemma gives us the result.
\end{proof}
\begin{cor}
    For $M \in \Sd^n\operatorname{-mod},$
    \[Me_n = 0 \iff \Ind{n}{n+1}(M)e_{n+1}=0\]
\end{cor}

\begin{lem}{\label{Pres}}
    For any $M \in \Sd^n\operatorname{-mod},$ we have an injective map of vector spaces 
    \[ \Res{n}{n-1}(M)e_{n-1} \hookrightarrow M \e \]
\end{lem}
    \begin{proof}
First, assume $P \in \Sd^n\operatorname{-proj}$. By the adjunction of \cite[Corollary 3.18]{wada} and the Schur functor:
\begin{align*}
    \Res{n}{n-1}(P)e_{n-1} & \cong \Hom_{\Sd^{n-1}}(e_{n-1}\Sd^{n-1}, \Res{n}{n-1}(Pe_{n-1}) \\
    &\cong \Hom_{\Sd^n}(\Ind{n-1}{n}(e_{n-1}\Sd^{n-1}), P)\\
    &\cong \Hom_{\Hh_n}(\Omega_n(\Ind{n-1}{n}(e_{n-1}\Sd^{n-1})), \Omega_n(P))\\
    &\cong \Hom_{\Hh_n}(\Omega_n(\Ind{n-1}{n}(e_{n-1}\Sd^{n-1})), P\eh{n})
\end{align*}
By \cite[Corollary 4.18(iv)]{wada}, this is then isomorphic to:
\begin{align*}
  \Hom_{\Hh_n}(\hind{n-1}{n}(e_{n-1}\Sd^{n-1}\eh{n-1}), P\eh{n}) &\cong \Hom_{\Hh_n}((\hind{n-1}{n}(V^{\otimes (n-1)}), P\eh{n}) \\
  &\inplus \Hom_{\Hh_n} (e_n \Sd^n \eh{n}, P \eh{n}) \\
  &\cong \Hom_{\Sd^n}(e_n\Sd^n, P) \\
  &\cong P e_n
\end{align*}  
    Therefore, there is a functor $\mathcal{Q}_{n-1}$ such that on projective modules \[\left( \oml{n-1} \circ \Res{n}{n-1} \right) \oplus \mathcal{Q}_{n-1} = \oml{n} \]
Applying the Five Lemma to the exact sequences obtained by applying these functors to a projective resolution for an arbitrary $\Sd^n$ module 
\[P_2 \rightarrow P_1 \rightarrow M \rightarrow 0\]
again gives us the result for general $M$.  Therefore, $\Res{n}{n-1}(M)e_{n-1} \hookrightarrow Me_n.$
\end{proof}

\begin{cor}
    Given $M \in \Sd^n\operatorname{-mod},$
    \[ Me_n = 0 \implies \Res{n}{n-1}(M)e_{n-1}=0\]
\end{cor}

\begin{thm}
    The induction and restriction functors of \cite{wada} induce functors 
\[ \Ind{n}{n+1}: \Sl^n\operatorname{-mod} \rightarrow \Sl^{n+1}\operatorname{-mod}\]
\[ \Res{n+1}{n}: \Sl^n\operatorname{-mod} \rightarrow \Sl^{n-1}\operatorname{-mod}\]
    compatible with the sequence of quotient functors in \cref{th:A}.
 \end{thm}

\subsection{Categorical action and crystal graphs}

\def\HEIS{\mathcal{H}eis}
\def\K{\Bbbk}
\newcommand{\slehat}{\mathfrak{\widehat{sl}}_e}
\newcommand{\slinfty}{\mathfrak{{sl}}_{\infty}}

One technique that has proven very powerful in recent decades is that of organizing functors like the induction and restriction into the actions of monoidal categories or 2-categories.  

Two types of such actions will be relevant for us:
\begin{enumerate}
	\item a {\bf quantum Heisenberg action} on a category $\mathcal{C}$ is a module category structure over the monoidal category $\HEIS_{\ell}(z,t)$ for a fixed level $\ell$ defined in \cite{brundanDefinitionQuantum2020}.  This is a formalization of the notion of having induction and restriction type functors which carry the action of finite Hecke algebras, Jucys--Murphy elements and satisfy a Mackey theorem.  
	\item a {\bf categorical Kac--Moody action} on a category $\mathcal{C}$ is an action of the Kac--Moody 2-category $\mathfrak{U}(\mathfrak{g})$ defined by Khovanov--Lauda \cite{khovanovCategorificationQuantum2010} and Rouquier \cite{Rou2KM} for a given root datum whose total category is $\mathcal{C}$.  While the definitions in the papers above were not obviously identical, later work of   \cite{Brundandef} showed that they coincide.  This formalizes the notion of breaking the category into a sum $\mathcal{C}\cong \oplus_{\mu}\mathcal{C}_{\mu}$ corresponding to the weight spaces of a $\mathfrak{g}$-action, and defining functors $E_i,F_i$ corresponding to the Chevalley generators.  
\end{enumerate}

The definition of these notions is quite complex and we will not need to use it in a deep way, so we refer to the readers to the references above for further details.  We will largely use them as a way of transporting well-known results for symmetric groups and type A Hecke algebras without having to repeat the proofs.    As in comparing with the conventions of \cite{LNX}, to compare with \cite{brundanHeisenbergKacMoody2020}, we need to choose square roots of $q$ and $Q$.  
\begin{lem}
	The categories $\bigoplus_{n}\Sl^n\mmod$ and $\bigoplus_{n}\Sd^n\mmod$ carry compatible quantum Heisenberg actions of level 2 by $\HEIS_2(q^{-1/2}-q^{1/2},\sqrt{-1/Q})$ for all choices of $q,Q\in \K$ with $E=\Res{}{},F=\Ind{}{}$.  
\end{lem}
The most important part of this action is the endomorphism $x\colon \Res{n}{n-1}\to \Res{n}{n-1}$ induced by multiplication by the Jucys--Murphy element $L_n$, and its mate $x^*\colon \Ind{n}{n-1}\to \Ind{n}{n-1}$.

\begin{proof}
	The analogous result for modules over the Hecke algebra is proven in 
\cite[Th. 6.3]{brundanDefinitionQuantum2020}; in the notation of that paper, we have $f(u)=u^2+(1-Q)u-Q$, so $t=\sqrt{-Q}$.  

This action extends to $\bigoplus_{n}\Sl^n\mmod$ and $\bigoplus_{n}\Sd^n\mmod$ by a standard argument (see, for example \cite[Th. 5.1]{ShanCrystal} and \cite[Th. 7.1]{brundanDefinitionQuantum2020}):  since the Schur functor $\Omega_n$ is fully-faithful on projectives, we can define the quantum Heisenberg action on $\bigoplus_{n}\Sd^n\operatorname{-proj}$ or $\bigoplus_{n}\Sl^n\operatorname{-proj}$ uniquely by the assumption that it commutes with $\Omega$.  On the other hand, any module over $\Sd^n$ or $\Sl^n$ can be presented as a cokernel of a map between these projectives.  By naturality and the uniqueness of this presentation up to homotopy, this extends the quantum Heisenberg action to the category of all modules $\bigoplus_{n}\Sl^n\mmod$ or $\bigoplus_{n}\Sd^n\mmod$.  
\end{proof}
For us, the primary significance of this result is that we can use it to define a Kac--Moody categorical action on $\bigoplus_{n}\Sl^n\mmod$ and $\bigoplus_{n}\Sd^n\mmod$.  
Consider the subset $U=Qq^{\mathbb Z}\cup -q^{\mathbb Z}$ of $\K$.  We make this into a directed graph by adding an arrow $u\to q^{-1}u$ for all $u\in I$. Let $\mathfrak{g}_{U}$ be the associated Kac--Moody algebra; in the notation of \cite{brundanHeisenbergKacMoody2020}, this would be denoted $\mathfrak{sl}'_{U}$.  This makes $U$ into a directed graph, which can take 4 possible forms:
\begin{enumerate}
	\item If $q$ is a primitive $e$th root of unity and $Q\in -q^{\mathbb Z}$, then $U$ is a single $e$-cycle and $\mathfrak{g}_{U}\cong \slehat$
	\item If $q$ has infinite multiplicative order and $Q\in -q^{\mathbb Z}$, then $U$ is a single copy of the $A_{\infty}$ Dynkin diagram and $\mathfrak{g}_{U}\cong \slinfty$.
		\item If $q$ is a primitive $e$th root of unity and $Q\notin -q^{\mathbb Z}$, then $U$ is  disjoint union of two $e$-cycles and $\mathfrak{g}_{U}\cong \slehat\oplus\slehat$
	\item If $q$ has infinite multiplicative order and $Q\notin -q^{\mathbb Z}$, then $U$ is a single copy of the $A_{\infty}$ Dynkin diagram and $\mathfrak{g}_{U}\cong \slinfty\oplus\slinfty$.
\end{enumerate}  
Consider the subfunctors
\begin{align*}
	 E_u(M)&=\{ m\in \Res{n+1}{n}(M) \mid (x-u)^N m=0 \text{ for all }N\gg 0\}\\ 
	 F_u(M)&=\{ m\in \Ind{n}{n+1}(M) \mid (x^*-u)^N m=0 \text{ for all }N\gg 0\}.
\end{align*}  
By \cite[Prop 3.7]{jf}, we can see that $E_u\neq 0$ if and only if $u\in U$, and similarly with $F_u$, so 
\[E(M)= \bigoplus_{u\in U} E_u(M) \qquad F(M)=\bigoplus_{u\in U} F_u(M).\]
By \cite[Th. A]{brundanHeisenbergKacMoody2020}, it follows immediately that:
\begin{lem}
	The functors $E_u$ and $F_u$ induce an action of the Kac--Moody 2-category $\mathfrak{U}(\mathfrak{g}_U)$ on $\bigoplus_{n}\Sl^n\mmod$ and $\bigoplus_{n}\Sd^n\mmod$. 
\end{lem}
This, in turn, has an important consequence for us that will give us a great deal more control over the simple modules over $\Sl^n$.  Let us record here the consequences of this action based on earlier work.  
\begin{samepage}

\begin{lem}\hfill
\begin{enumerate}
    \item 
	The sets of simples $\Simp(\Sl^n)$ and $\Simp(\Sd^n)$ for all $n\geq 0$ form a crystal graph for the  representation of $\mathfrak{g}_{U}$ on the Grothendieck group of these categories.  The inclusions $\Simp(\Hh^n)\hookrightarrow\Simp(\Sl^n)\hookrightarrow\Simp(\Sd^n)$ are maps of crystals.
 \item The simple $\tilde{e}_u(L)$ is isomorphic to the unique simple quotient and  unique simple submodule of $E_u(L)$ and $\tilde{f}_u(L)$  is isomorphic to the unique simple quotient and  unique simple submodule of $F_u(L)$.  
 \item 
	Under the bijection of the set of simples for $ \Sd^n(\Lnm)$ with the set $\Pi_n$ of all bipartitions of $n$, the Kashiwara operators act as follows:
 \begin{itemize}
     \item Consider the set $\mathcal{AR}_u$ of boxes of residue $u$ in the diagram of the bipartition $(\lambda^{(1)},\lambda^{(2)})$ which are addable or removable. 
     \item Order this set with all boxes in $\lambda^{(1)}$ larger than those in $\lambda^{(2)}$, and with those further to the right higher within components.  That is, we have $(r,c,\ell)>(r',c',\ell')$ if $\ell <\ell'$ or if $\ell=\ell'$ and $c>c'$.   Note that addable boxes are ordered according to the induced dominance order on Young diagrams obtained when they are added.
     \item List the boxes in $\mathcal{AR}_u$ in decreasing order, and replace all removable boxes with open parentheses ``$($'' and all addable boxes with close parentheses ``$)$''.  The Kashiwara operator $\tilde{f}_i$ acts by adding the lowest (in our order) addable box whose parenthesis is uncanceled, and $\tilde{e}_i$ acts by removing the highest removable box whose parenthesis is uncanceled.
 \end{itemize}
 We can state this more conceptually by saying that if we only consider adding and removing boxes with residue $u$, then we obtain a tensor product of copies of the 2-dimensional representation of $\mathfrak{sl}_2$ for each element of $\mathcal{AR}_u$, tensored in our chosen order.  
\end{enumerate}
	
\end{lem}
	
\end{samepage}

\begin{proof}
    Statements (1) and (2) are consequences of \cite[Th. 4.31]{brundanCategoricalActions2016}.  
    
    Statement (3) follows from the fact that the categorical action on $\bigoplus_{n}\Sd^n\mmod$ is a ``highest weight categorification'' in the sense of \cite{losevHighestWeight2013}.  This is confirmed for the category $\mathcal{O}$ over a cyclotomic Cherednik algebra in \cite[\S 4.2]{losevHighestWeight2013}; by \cite[Cor. 3.11]{WebRou} for any integer $d$, we can always choose parameters of the Cherednik algebra such that the categories $\Sd^n$ for $n\leq d$ are equivalent to category $\mathcal{O}$ with these parameters compatibly with the functors $\Ind{n}{n-1}, \Res{n}{n-1}$. Thus, the description of the crystal above follows from \cite[Cor. 5.6]{WebRou}, since the structure we have described above is the $\boldsymbol{\vartheta}$-weighted crystal structure for an appropriate choice of parameters.
\end{proof}
\begin{example}
\ytableausetup{boxsize=1em}
    Let $q=Q=-1$.  In this case $U=\{1,-1\}$, and  the first few levels of the crystal graph are shown in Figure \ref{fig:ex-1}.
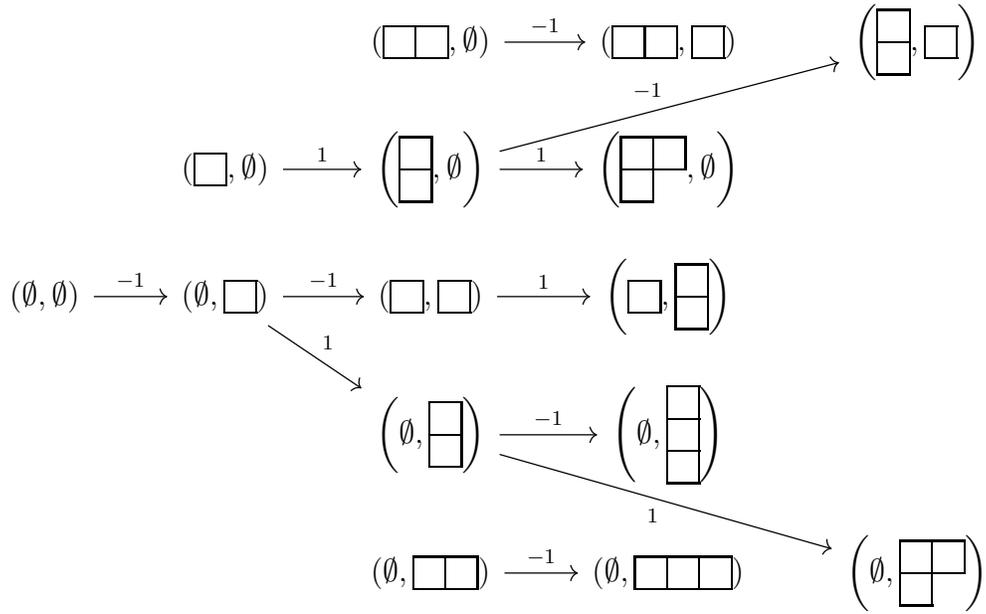
\begin{figure}[htb]
    \centering
\[\begin{tikzcd}[row sep=small]
	&& {\left( \ydiagram{2},\emptyset\right)} & {\left(\ydiagram{2},\ydiagram{1}\right)} & {\left(\ydiagram{1,1},\ydiagram{1}\right)} \\
	& {\left(\ydiagram{1},\emptyset\right)} & {\left(\ydiagram{1,1},\emptyset\right)} & {\left(\ydiagram{2,1},\emptyset\right)} \\
	{\left( \emptyset,\emptyset\right)} & {\left(\emptyset,\ydiagram{1}\right)} & {\left(\ydiagram{1},\ydiagram{1}\right)} & {\left(\ydiagram{1},\ydiagram{1,1}\right)} \\
	&& {\left(\emptyset,\ydiagram{1,1}\right)} & {\left(\emptyset,\ydiagram{1,1,1}\right)} \\
	&& {\left(\emptyset,\ydiagram{2}\right)} & {\left(\emptyset,\ydiagram{3}\right)} & {\left(\emptyset,\ydiagram{2,1}\right)}
	\arrow["{-1}", from=3-1, to=3-2]
	\arrow["{-1}", from=3-2, to=3-3]
	\arrow["1", from=2-2, to=2-3]
	\arrow["1", from=3-2, to=4-3]
	\arrow["{-1}", from=1-3, to=1-4]
	\arrow["{-1}", from=5-3, to=5-4]
	\arrow["1"', from=4-3, to=5-5]
	\arrow["1", from=2-3, to=2-4]
	\arrow["1", from=3-3, to=3-4]
	\arrow["{-1}", from=4-3, to=4-4]
	\arrow["{-1}", from=2-3, to=1-5]
\end{tikzcd}\]
    \caption{Crystal structure in the case where $q=Q=-1$.}
    \label{fig:ex-1}
\end{figure}
For example, if we consider ${\left(\emptyset,\ydiagram{2}\right)} $, this has two addable boxes with residue $-1$, which are $(1,1,1)>(3,1,2)$, so the Kashiwara operator acts by adding the lower of these; our sequence of parentheses is $))$.  On the other hand, $(2,1,2)>(1,2,2)$ are a removable and an addable box with residue $1$, so our sequence of parentheses is $()$; there are no uncanceled parentheses, so $\tilde{e}_1,\tilde{f}_1$ both kill this Young diagram.

On the other hand, let us consider the case of the Counterexample \ref{sec:counterexample} where $Q=-q$ for $q$ generic (for the portion we show, $q^2\neq 1,q^3\neq 1$ suffices). In this case the first few levels of the crystal graph are shown in Figure \ref{fig:ex-2}.
\begin{figure}[ht]
    \centering
\[\begin{tikzcd}[row sep=small]
	&&& \cdots \\
	&& {\left(\ydiagram{2},\emptyset\right)} & \cdots \\
	& {\left(\ydiagram{1},\emptyset\right)} & {\left(\ydiagram{1},\ydiagram{1}\right)} & {\left(\ydiagram{1,1,1},\emptyset\right)} & \cdots \\
	{\left(\emptyset,\emptyset\right)} && {\left(\ydiagram{1,1},\emptyset\right)} \\
	& {\left(\emptyset,\ydiagram{1}\right)} & {\left(\emptyset,\ydiagram{1,1}\right)} & {\left(\ydiagram{2,1},\emptyset\right)} & \cdots \\
	&& {\left(\emptyset,\ydiagram{2}\right)} & \cdots \\
	&&& \cdots
	\arrow["{-1}", from=4-1, to=5-2]
	\arrow["{-q^{-1}}", from=5-2, to=5-3]
	\arrow[from=2-3, to=1-4]
	\arrow[from=6-3, to=7-4]
	\arrow["{-q^{2}}"', from=4-3, to=5-4]
	\arrow[from=3-3, to=2-4]
	\arrow[from=5-3, to=6-4]
	\arrow["{-q}"{description}, from=4-1, to=3-2]
	\arrow["{-1}", from=3-2, to=3-3]
	\arrow["{-q}", from=5-2, to=6-3]
	\arrow["{-q^2}", from=3-2, to=2-3]
	\arrow["{-q^{-1}}", from=4-3, to=3-4]
	\arrow[from=3-4, to=3-5]
	\arrow[from=5-4, to=5-5]
\end{tikzcd}\]
    \caption{Crystal structure in the case of Counterexample \ref{sec:counterexample} where $Q=-q$ for $q$ generic.}
    \label{fig:ex-2}
\end{figure}

\end{example}

This is useful for us since it implies that:
\begin{cor}\label{LNXcl}
The subsets of  \LNX{o} and \LNX{e} bipartitions and their complements are both closed under the action of Kashiwara operators.  In particular, any crystal graph component with non-empty intersection with $\LnmB$  for $m$ odd consists of \LNX{o} bipartitions, and similarly for $m$ even and \LNX{e} bipartitions.
\end{cor}
Note that this does not mean that the \LNX{m} bipartitions for any fixed value of $m$ are closed under crystal operators; this is not, in fact, the case.

Thus, for large $m$ we can determine the set of simples over $\Sl^n(m)$ for all $n$ by testing a single simple in each component of the crystal graph.  

In the first example above, only the simples in the largest component are Kleshchev, whereas our rank 2 calculation shows that all the components appearing consist of \LNX{o} bipartitions.  There is a component of bipartitions which are not \LNX{o} generated by ${\left(\ydiagram{2,2},\emptyset\right)}$ by Corollary \ref{formula}.

In the second example, the component of the crystal graph that contains $\left(\ydiagram{1,1},\emptyset\right)$ is not \LNX{o}.  

\section{Proof of the main theorems}
We can factor the polynomial $f_n(Q,q)$ as a product $ f_n(Q,q) = 
\pol{n}(Q,q) \cdot b_n(Q,q)$ where
\[ 
\pol{n}(Q,q) := \prod_{i=2-\lfloor \frac{n}{2}\rfloor}^{n-1} (Q+q^i) \qquad b_n(Q,q):= \prod_{-(n-1)}^{i=1-\lfloor \frac{n}{2} \rfloor}(Q+q^i). \]

We can thus rephrase the equivalence of (3) and (4) in \cref{th:B} as:
\begin{thm}\label{th:quasi-h}
    Assume $m$ large odd. The algebra $\Sl^n(m)$ is quasi-hereditary if and only if $\pol{n}(Q,q) \neq 0$.
\end{thm}
\begin{proof}
%
We can assume that $n$ is minimal such that we have a counter-example to this result; we have already confirmed that it is correct for $n\leq 2$.  

If $f_n(Q,q) \neq 0$, then \cite[Cor. 6.1.1]{LNX} shows that $\Sl^n(m)$ is quasi-hereditary, so we need only consider the case where $f_n(Q,q)=0$.  
  
{\bf ``Only if'' direction:}. Assume that  $\pol{n}(Q,q) = 0$, so $Q=-q^k$ for some $k\in [\frac{4-n}{2},n-1]$.
By \cref{formula}, the bipartitions $((1^n),\emptyset)$ is not \LNX{o} if $k>0$ and  $\left( \left( \lceil \frac{n}{2} \rceil, \lfloor \frac{n}{2} \rfloor \right), \emptyset \right)$ is not \LNX{o} if $k\leq 0$. By \cref{cor:bipartitions}, this implies that $\Sl^n$ is not quasi-hereditary.

{\bf ``If'' direction:} Now, assume that $\pol{n}(Q,q) \neq 0$.  First, note that the assumption that $\pol{n}(Q,q)\neq 0$ and $b_n(Q,q)=0$ can only hold if $q$ is not a root of unity of order $e<n$.  If $q$ were such a root of unity and $Q=-q^k$ for some $k\in [1-n,1-\lfloor \frac{n}{2} \rfloor]$, then $k\equiv k'\pmod e$ for some $k'\in [1,e]$, and $\pol{n}(Q,q) = 0$, contradicting our assumption.  Thus, we must have $e\geq n$ if $q^e=1$.

 Assume that there is a bipartition $(\lambda^{(1)},\lambda^{(2)})$ of $n$ which is not \LNX{o}.
If $\tilde{e}_u(\lambda)\neq 0$, then the resulting bipartition is not \LNX{o} by \cref{LNXcl} and $\pol{n-1}(Q,q)\neq 0$ since $\pol{n-1}(Q,q)$ divides $\pol{n}(Q,q)$.  Thus,  minimality of $n$ implies that $\lambda=(\lambda^{(1)},\lambda^{(2)})$ is highest weight for the crystal, i.e. $\tilde{e}_u(\lambda)=0$ for all $u\in U$. 
 By \cref{klimits} and the assumption that $\pol{n}(Q,q) \neq 0,$ we have that $\lambda^{(2)}\neq \emptyset$.  On the other hand, as observed before, we must have $\lambda\notin \LnmB$, so $\lambda^{(1)}\neq \emptyset$ (in fact, it must have at least 2 rows).  
    
  Now, consider the removable box $(r,c,2)$ in $\lambda^{(2)}$ in the last row of $\lambda^{(2)}$, i.e. that maximizes $r$ and minimizes $c$.  Let $u=-q^{c-r}$ be the residue of this box.  By the highest weight assumption, the parenthesis corresponding to this removable box must be canceled, and so there must be an addable box of the same residue with $c$ lower.  Of course, this addable box must be $(r+1,1,2)$, since that is the only one lower in our order.  Since these boxes have the same residue, we must have that $e$ divides $c$, but this means that $e\leq c<n$.  But this contradicts the argument we gave above that $e\geq n$, completing the proof. 
\end{proof}

\begin{thm}\label{th:quasi-h-even}
    Assume $m$ is large even.  The algebra $\Sl^n(m)$ is quasi-hereditary if and only if $f_n(Q,q) \neq 0$.
\end{thm}
\begin{proof}
    The ``if'' direction is precisely \cite[Cor. 6.1.1]{LNX}, so we need only prove the ``only if'' direction.  

    If $f_n(Q,q)=0$, then we have $Q=-q^k$ for some $k\in [-n+1,n-1]$.  In this case, if $k\geq 0$, the bipartition $((1^k),\emptyset)$ is not  \LNX{e} by \cref{formula-even}(1).  The crystal graph component of this element contains a partition of size $n$, so $\Sl^n(m)$ is not quasi-hereditary.   If $k<0$, a similar argument follows with $((-k),\emptyset)$, using \cref{formula-even}(2).
\end{proof}

\begin{proof}[Proof of \cref{th:B}]
The equivalence of (1), (2), and (3) follows from \cref{cor:corner-algebra}.  The equivalence of (2) and (4) is precisely \cref{th:quasi-h,th:quasi-h-even}.  
\end{proof}

\begin{prop}\label{B-even}
    If $q$ is not a root of unity of order $\leq n$, 
    \begin{enumerate}
    \item the \LNX{o} bipartitions are those in the crystal graph component of a bipartition of the form $((m),\emptyset)$ for some $m$.  
    \item  the \LNX{e} bipartitions are precisely the Kleshchev bipartitions, that is, the crystal graph component of $(\emptyset,\emptyset)$. 
    \end{enumerate}
\end{prop}
\begin{proof}
{\bf ``If'' direction}: Recall that $((m),\emptyset)$ is always \LNX{o} (\cref{lem:LB-LNX}), so the closure under crystal operations (\cref{LNXcl}) implies that if $(\lambda^{(1)},\lambda^{(2)})$ is in the crystal graph component of $((m),\emptyset)$, it is \LNX{o}.  This proves the ``if'' direction of (1).  The same argument establishes the ``if'' direction of (2) since $(\emptyset,\emptyset)$ is \LNX{e}.

{\bf ``Only if'' direction:}  Now, assume that $\lambda$ is \LNX{o}.  
As $q$ is not a root of unity of order $\leq n$,  there is at most one addable or removable box of a given residue in $\la^{(2)}$.  Thus, any removable box in $\la^{(2)}$ can be removed by the Kashiwara operator of the corresponding residue.  It follows that if $\lambda$ is \LNX{o}, the highest weight object in the component of $\lambda$ must be of the form $\mu=(\mu^{(1)},\emptyset)$.  Note that if $\mu$ has any removable box, it must be of residue $-1$, since any other removable box is not canceled by an addable box in $\emptyset$, contradicting the highest weight hypothesis.  Thus, we must have $\mu=((k+r)^r,\emptyset)$ for $r\geq \max(0,-k)$ where $Q=-q^k$.  If $r>1$, then the removable box $(k+r,r)$ is not in the first row, so by  \cref{formula}, the bipartition $\mu$ is not \LNX{o}, showing that we must have $r\leq 1$, as desired.  This completes the proof of statement (1).

On the other hand, if we assume that $\la$ is \LNX{e}, then \cref{formula-even} shows that $\mu=(\emptyset,\emptyset)$, since any box would show it was not \LNX{e}.  
\end{proof}

If $q$ is a root of unity of order $e\leq n$, then \cref{B-even} is manifestly false as stated: The component of the crystal graph component containing $(\emptyset,(e))$ consists of  \LNX{o} and \LNX{e} bipartitions, but does not contain any of the form $(\lambda, \emptyset)$.  As mentioned in the introduction, we expect that this can be repaired by adding Heisenberg crystal operators as defined in \cite{losevSupportsSimple2021}.

It's natural to wonder how these results extend to smaller values of $m$.  As our discussion at the end of Section \ref{sec:counterexample} shows, $\Sl^n(m)$ can be quasi-hereditary for $m<2n+1$ even if the condition (4) of \cref{th:B} holds.  Since we must have $r>n$ for the cell module $\Csl{((1^n),\emptyset)}$ to be non-zero over $\Sl^n(2r-1)$, one natural guess is that:
\begin{conj} 
The algebra $\Sl^n(2r-1)$ is quasi-hereditary if and only  $Q\neq -q^k$ for all $k$ satisfying $\frac{4-n}{2}\leq k<\min(n,r)$.
\end{conj}

\subsection{Quasi-hereditarity for blocks}

We'll assume throughout this subsection that $m$ is large.  In this case, we can compute the block structure of $\Sl^n(m)$ by comparison with that of $\Sd^n$ and $\Hh_n$.  In any cellular algebra, every cell module lives in a single block (even though in some cases they are not indecomposable). The block of a cell module over $\Sd^n$ or over $\Hh_n$ is determined by the function $b_\lambda\colon \K\to \Z_{\geq 0}$ sending $u\in \K$ to the number of boxes $(r,c,\ell)$ in the diagram of $\la$ such that $\operatorname{res}(r,c,\ell)=u$ (where the residue function is defined in \cref{content-def}).  
Two bipartitions with the same function $b_{\la}$ are called {\bf residue equivalent}.

By \cite[Th. 2.11]{Lyle2007}, we can write $\Sd^n(\Lnm)\cong \bigoplus_{b}\Sd(b)$ as a sum of block algebras corresponding to the functions $b$ of the form $b_{\la}$ for at least one bipartition $\la$.  We have an associated decomposition $\Sl^n(m)\cong \oplus_{b}\Sl(b)$ where $\Sl(b)=\e\Sd(b)\e$.   
\begin{prop}\label{prop:Sl-blocks}
The algebras $\Sl(b)$ are indecomposable, that is, they are the block algebras of $\Sl^n(m)$.  Consequently, blocks of $\Sd^n,\Sl^n(m)$ and $\Hh_n$ are in canonical bijection.
\end{prop}
\begin{proof}
    The proof is effectively the same as the first half of \cite[Prop. 2.3]{Lyle2007}: if we have two simple $\Sd(b)$ modules $\Ssl{\lambda}$ and $\Ssl{\la'}$, then the cell modules $\Ch{\lambda}$ and $\Ch{\la'}$ must lie in the same block of $\Hh_n$-modules.  By \cite[(3.9.8)]{grahamCellularAlgebras1996}, we must have a chain of bipartitions $\xi=\xi_0,\xi_1,\dots, \xi_d=\la'$ and Kleshchev bipartitions $\nu_1,\dots, \nu_n$ such that $\Sh{\nu_i}$ is a composition factor in both $\Ch{\xi_{i-1}}$ and $\Ch{\xi_{i}}$.  By the exactness of multiplying by $\eh{n}$, it is likewise true that $\Ssl{\nu_i}$ is a composition factor of  $\Csl{\xi_{i-1}}$ and $\Csl{\xi_{i}}$.  Thus,  $\Ssl{\lambda}$ and $\Ssl{\la'}$ must lie in the same block.
\end{proof}

We can now turn to the question of whether these block algebras of $\Sl^n(m)$ are quasi-hereditary.  By \cref{lem:corner-cell}, the block algebra $\Sl(b)$ will be cellular and by \cref{cor:corner-algebra}, it will be quasi-hereditary if and only the idempotent $\e$ induces a Morita equivalence to $\Sd(b)$, that is, if
every bipartition $\la$ satisfying $b_{\la}=b$ is \LNX{m}.  

In order to attack this problem, it is useful to use the Weyl group action on blocks.  For each function $b$, we have the statistic $\al_u^{\vee}(b)=b(qu)+b(q^{-1}u)-2b(u)+\delta_{u,-1}+\delta_{u,Q}$, given by the number of addable boxes of residue $u$ minus the number of removable boxes of residue $u$ for any bipartition whose cell module lies in this block.  If $k=\al_i^{\vee}(b)<0$, then the iterated Kashiwara operator $\tilde{e}_i^k$ gives a bijection between bipartitions with residue function $b$ and those with residue function $b'(u)=b(u)-k\delta_{u,i}$.  Of course, we have $\al_i^{\vee}(b')=-k$; in fact, $b$ and $b'$ correspond to weight spaces of $\mathfrak{g}_U$ conjugate under the Weyl group of this Lie algebra.  Furthermore, the block algebra $\Sl(b)$ is quasi-hereditary if and only if $\Sl(b')$ is, since this bijection sends \LNX{o} bipartitions to \LNX{o} bipartitions (and similarly with $m$ even).  This process reduces the number of boxes in the bipartitions at each step, and so must terminate.  Therefore, we can assume that $\al_u^{\vee}(b)\geq 0$ for all $u\in U$, i.e. that $b$ corresponds to a dominant weight of $\mathfrak{g}_U$.  In this case, we call the block algebra $\Sl(b)$ {\bf dominant}.

We'll use a standard consequence of the convexity of weight diagrams for representations of $\mathfrak{g}_U$ (see \cite[Lem. 3.6]{websterRoCKBlocks2023}) for any subcrystal $\mathfrak{C}$ inside the set of all bipartitions. 
\begin{lem}\label{lem:add-hooks}
    Assume $b'$ is dominant and that $b'(u)\geq b(u)$ for all $u$. If $\mathfrak{C}$ contains a bipartition $\la$ with $b_\la=b$, then there is a bipartition $\mu\in \mathfrak{C}$ with $b_{\mu}=b'$ which contains $\la$ in its diagram.
 \end{lem}
 \begin{proof}
     While this follows from convexity, it's simple to give a ``hands-on'' proof by induction on the difference between the number of boxes.  If we let $b''=b'-b$, we can write this difference as $\sum_{u\in U}b''(u).$  The strategy of the proof is to show that there is some $u\in U$ where $b''(u)>0$ and $\al_u^{\vee}(b)>0$.  In this case, we have more addable boxes than removable boxes with residue $u$, so the Kashiwara operator $\la'=f_u(\lambda)$ gives a bipartition in $\mathfrak{C}$ whose diagram contains that of $\la$, and still satisfies $b_{\la'}(u)\leq b'(u)$.  By induction, there is a bipartition $\mu\in \mathfrak{C}$ with $b_\mu=b'$ containing the diagram of $\la'$, and thus that of $\la$.
     
     {\bf Existence of  $u\in U$ where $b''(u)>0$ and $\al_u^{\vee}(b)>0$:} Let $a$ be the maximum value of the function $b''$ on $U$, and $v$ a point where it is obtained.  Either (1) $b''(q^kv)=a$ for all $k\in\Z$, or (2) we can assume that $b''(q^{-1}v)<a$ and $b''(qv)\leq a$.  
     
     In case (1), the set $q^{\Z}v$ must be finite and we can consider the sum $\ell=\sum_{v'\in q^{\Z}v}\al_{v'}^{\vee}(b)$.  This is the difference of the total number of addable and the total number of removable boxes in the components of the partition whose residues lie in $q^{\Z}v$.  Thus $\ell$ is just the number of such components.  That is, we have $\ell=1$ if exactly one of $-1\in q^{\Z}v$ or $Q\in q^{\Z}v$ holds, and $\ell=2$ if both do.  Thus,  for some $v'\in q^{\Z}v$, we have $\al_{v'}^{\vee}(b)=\al_{v'}^{\vee}(b')>0$.  
     
     In case (2), we already have $\al_{v}^{\vee}(b)>\al_{v}^{\vee}(b')\geq 0$.  In either case, we can complete the proof by the induction argument discussed above.   
 \end{proof}

    
Let $K=\{k\in \Z \mid Q=-q^{k}\}\subset \Z$.\begin{prop}\label{prop:blocks-qh}
  Assume $m$ is large odd.
    A dominant block algebra $\Sl(b)$ is not quasi-hereditary if and only if one of the conditions below holds:
    \begin{enumerate}
        \item For some $k\in K$ with $k>0$, we have  $b(u)\geq b_{((1^{k-1}),\emptyset)}(u)$ for all $u\in U$ .
        \item For some $-k\in K$ with $k\geq 0$,  we have  $b(u)\geq b_{((k^2),\emptyset)}(u)$ for all $u\in U$ .
    \end{enumerate}
    By Lemma \ref{lem:add-hooks}, we can rephrase these conditions equivalently as:
        \begin{enumerate}
        \item[($1'$)] Some bipartition $\lambda$ with $b_{\lambda}=b$ contains $((1^{k-1}),\emptyset)$ as a subdiagram.
        \item[($2'$)] Some bipartition $\lambda$ with $b_{\lambda}=b$ contains $((k^2),\emptyset)$ as a subdiagram.
    \end{enumerate}
\end{prop}
Note that the dominant condition is necessary here.  There will be many semi-simple blocks which satisfy the conditions (1) and (2), but the only dominant semi-simple block algebra is $\Sl^0$, the rank 0 Schur algebra (which corresponds to the function $b=0$).

\begin{proof}
{\bf ``Only if'' direction:}    Since we are considering a dominant block, if condition (1) holds, then by Lemma \ref{lem:add-hooks}, the crystal graph component of $((1^{k-1}),\emptyset)$ contains a bipartition $\lambda$ with $b_{\lambda}(u)=b(u)$.   By \cref{formula},  all bipartitions in this crystal graph component are not \LNX{o}, so the block is not quasi-hereditary.

    Similarly, if (2) holds, then the same argument applied to the component of $((k^2),\emptyset)$ shows that the block is not quasi-hereditary.

  {\bf ``If'' direction:}  On the other hand, assume $\Sl(b)$ is not quasi-hereditary.  
  There must be a non \LNX{o} bipartition $\mu$ with $b_{\mu}=b$, whose crystal graph component has highest weight bipartition $(\la^{(1)},\la^{(2)})$.  If $\la^{(2)}=\emptyset$, then \cref{formula} shows that $\la^{(1)}$ must contain a box $(a,b)$ with $k=a-b\in K$, and $a>1$.  If $k>0$, then this shows that (1) holds and if $k\leq 0$, then (2) holds.

    Thus, we need only consider the case where $\la^{(2)}\neq \emptyset$. As argued in the proof of \cref{th:quasi-h},  we must have a minimal positive integer $e$ such that $-U$ is the $e$th roots of unity.  Furthermore, the smallest part of $\la^{(2)}$ must be of length $ae$ for $a\geq 1$, showing that $b(u)\geq a$ for all $u\in U$.  
    
    Since $Q\in U$, there must be some $k\in K\cap [1,e]$.  Since $b_{((1^{k-1}),\emptyset)}(u)$ only has values in $\{0,1\}$, this shows that condition (1) holds. 
\end{proof}
Note that the same proof easily extends to the even case:
\begin{prop}
	  Assume $m$ is large even. A dominant block algebra $\Sl(b)$ is not quasi-hereditary if and only if one of the conditions below holds:
    \begin{enumerate}
        \item For some $k\in K$ with $k>0$, we have  $b(u)\geq b_{(1^{k-1},\emptyset)}(u)$ for all $u\in U$ .
        \item For some $-k\in K$ with $k\geq 0$,  we have  $b(u)\geq b_{(k-1,\emptyset)}(u)$ for all $u\in U$.
    \end{enumerate}
\end{prop}
\input{cherednik.tex}

\printbibliography
\end{document}

%% file: cherednik.tex
\section{Connection to rational Cherednik algebras}
\subsection{Background on rational Cherednik algebras}

The DJM Schur algebra $\Sd^n(\Lnm)$ for $m$ large is not the only quasi-hereditary cover of the modules over the Hecke algebra $\mathcal{H}_n$; it is a special case of a larger family of covers given by category $\cO$ over Cherednik algebras.  Let us fix some notation here, since the literature on rational Cherednik algebras has a great multiplicity of conventions for parameters.

Let $S_1$ be the set of complex reflections in
the type $B$ Weyl group $W_n$ that
switch two coordinate subspaces (and thus are conjugate to $s_1$) and $S_0$ the set which preserve the
coordinate subspaces (and thus are conjugate to $s_0$).  For each such reflection, let $\al_s$ be a
linear function vanishing on the hyperplane $\ker(s-1)$, and $\al_s^\vee$ a vector
spanning the image of $s-1$ such that $\langle \al_s^\vee,\al_s\rangle =2$.  
Let \[\omega_s(y,x)=\frac{\langle y,\al_s\rangle \langle \al_s^\vee,x\rangle }{\langle
  \al_s^\vee,\al_s\rangle }=\frac{\langle y,\al_s\rangle \langle \al_s^\vee,x\rangle }{2 }\]
 Let $T(\C^n\oplus
(\C^n)^*)$ denote the $\C$-linear tensor algebra on the vector space $\C^n\oplus
(\C^n)^*$, and $T(\C^n\oplus
(\C^n)^*)\# W_n$  its smash product (i.e. semi-direct product) with the group algebra $\C W_n$.  
\begin{defn}
	The rational Cherednik algebra is the quotient of the algebra $T(\C^n\oplus
(\C^n)^*)\# W_n$ by the relations for $y,y'\in \C^n,x,x'\in (\C^n)^*$:
\[[x,x']=[y,y']=0\]\[ [y,x]=\langle y,x\rangle
-\sum_{s\in S_1}2h\omega_s(y,x)  s-\sum_{s\in S_0} 2H\omega_s(y,x)s.\]
\end{defn}
If we let $H_1,H_0$ be the reflection hyperplanes for $s_1,s_0$, in the notation of \cite{RouqSchur} and  \cite{ginzburgCategoryRational2003}, we would write $h=h_{H_1,1}=-k_{H_1,1}, H=h_{H_0,1}=-k_{H_0,1}$.  

\begin{defn}
  Category $\cO$, which we denote $\cO^{h,H}_n$,
  is the full subcategory of modules over $\mathsf{H}$ which are
  generated by a finite dimensional subspace invariant under
  $\operatorname{Sym}(\C^n)\#\C[W_{n}]$ on which $\operatorname{Sym}(\C^n)$ acts
  nilpotently.  
\end{defn} 
The simple modules in this category are canonically labeled with representations of $W_{n}$, which in turn have a usual bijection with $\Lambda_n^+$, the bipartitions of $n$.  
By \cite[Th. 2.19]{ginzburgCategoryRational2003}, this category is highest weight, with a partial order that depends on $h,H$:
\begin{enumerate}
	\item Assign the boxes $(r,c,\ell)$ of the diagram of $\lambda$ the charged content 
\[\operatorname{con}(r,c,\ell)=	-h(c-r)+\frac{1}{2}H(\delta_{\ell,2}-\delta_{\ell,1}).\]
	\item Let $c_{\lambda}$ be the sum of these charged contents.
\end{enumerate} 
\begin{lem}
Category $\cO$ is highest weight with respect to the partial order on bipartitions where $\lambda\geq  \mu$ if there is a bijection between the boxes of the diagrams of $\lambda$ and $\mu$ such that:
\begin{enumerate}
	\item This bijection preserves $(Q,q)$-residue of boxes.
	\item If the box $(r,c,\ell)$ in $\la$ corresponds to $(r',c',\ell')$ in $\mu$, then $\operatorname{con}(r,c,\ell)-\operatorname{con}(r',c',\ell')\geq 0$ with equality only if $(r,c,\ell)=(r',c',\ell')$.   
\end{enumerate}
\end{lem}
The condition of equality of residues is equivalent to requiring that $\operatorname{con}(r,c,\ell)-\operatorname{con}(r',c',\ell')$ is an integer if $\ell=\ell'$ and in $\Z+\frac{1}{2}$ if $\ell\neq \ell'$.

Note that if $H<nh<0$, then the content of any box in $\lambda^{(1)}$ is greater than that of any in $\lambda^{(2)}$, so this order coarsens the dominance order of multi-partitions.  
\begin{proof}
The only difference from \cite[Th. 2.19]{ginzburgCategoryRational2003} is that our order is coarser than that given there.  This follows from \cite[Th. 4.11]{WebRou}: there can only be a homomorphism from the projective cover of $\Delta(\mu)$ to $\Delta(\lambda)$ if there is a cellular basis vector of shape $\lambda$ and type corresponding to $\mu$. This can only happen if (1-2) hold.    
\end{proof}

\begin{example}
Consider $n=2$; in this case, $W_2$ is the Klein 4-group, and we have 4 1-dimensional representations:
\begin{enumerate} 
	\item The trivial representation corresponds to $((2),\emptyset)$.
	\item The character $s_0\mapsto 1,s_1\mapsto -1$ corresponds to $((1,1),\emptyset)$.
	\item The character $s_0\mapsto -1,s_1\mapsto 1$ corresponds to $(\emptyset,(2))$.	
	\item The sign representation $s_0\mapsto -1,s_1\mapsto -1$ corresponds to $(\emptyset,(1,1))$.
\end{enumerate}
The highest weight structure is determined by the action of the Euler element \[\mathsf{eu}=y_1x_1+y_2x_2-h s_1-hs_0s_1s_0-Hs_0-Hs_1s_0s_1.\]  Since $x_i$ acts trivially on highest weight vectors, the $c$-function for this element is 
\[c_{((2),\emptyset)}=-2h-2H\qquad c_{((1,1),\emptyset)}=2h-2H\qquad c_{(\emptyset,(2))}=-2h +2H\qquad c_{(\emptyset,(1,1))}=2h +2H\] matching the description above.  Note that any order of these 4 representations is possible, based on different choices of $h,H$.
\end{example}

Let $q=\exp(2\pi i h)$ and $ Q=\exp(2\pi i H)$.  

\begin{thm}[\mbox{\cite[Th. 5.16]{ginzburgCategoryRational2003}}]
There is a functor $\KZ\colon \cO_{h,H}^n\to \mathcal{H}_n\mmod$ which is fully faithful on projectives.  Thus, this makes $\cO_{h,H}^n$ into a quasi-hereditary cover of $\mathcal{H}_n\mmod$.  
\end{thm}
By \cite[Th. 5.5]{RouqSchur}, the category $\cO_{h,H}^n$ only depends on the values $(Q,q)$ and the $c$-function order on $\Lambda_n^+$, and in fact when this order is a coarsening of dominance order (for example, when $H<nh<0$), we will have an equivalence $\cO_{h,H}^n\cong \Sd^n(\Lnm)\mmod $ for $m$ large by \cite[Cor. 3.11]{WebRou}.

\subsection{Schur restriction}
In this section, we will define a generalized KZ-functor which lands in the modules over Schur algebra $\Sl^n(m)$ instead of the Hecke algebra $\Hh_n$.  

Assume for the remainder of this section that:
\begin{itemize}
	\item[$(\dagger)$] for all $0<r\leq n$, the multipartition $((r),\emptyset)$ is maximal in our partial order.  
\end{itemize} 
\begin{lem}
	The condition $(\dagger)$ holds if and only if $rh\notin \Z_{>0}$ and $H+(r-1)h+\frac{1}{2}\notin \Z_{>0}$ for all $0<r\leq n$.
\end{lem}
\begin{proof}
	If $((1^{r'}),\emptyset)$ is not maximal, there is some $\mu'>(({r'}),\emptyset)$.  Since $\mu'\neq (({r'}),\emptyset)$, either $(2,1,1)$ or $(1,1,2)$ lies in its diagram.  Under our bijection, this box must correspond to one of the form $(r,1,1)$.  The difference of the contents is either $rh$ or $(r-1)h-H$, so one of these must be positive.  Since the boxes must have the same residue, this means that $rh\in \Z_{>0}$ or $H+(r-1)h+\frac{1}{2}\in \Z_{>0}$ 
	
	On the other hand, if $rh\in \Z_{>0}$, then $((r-1,1),\emptyset)>((r),\emptyset)$, and if $(r-1)h+H+\frac{1}{2}\in  \Z_{>0}$ then $((r-1),(1))>((r),\emptyset)$.  
\end{proof}
This condition will hold for all $n$ if $h,H$ are non-positive real numbers, or more generally have non-positive real parts, but will also hold for generic complex numbers independent of signs.  
The results of this section can be modified to work in other cases, as we will discuss later, but this case is best adapted to our conventions about Schur algebras.  
\begin{lem}
For $\alpha\in \Anm$, 
consider the reflection group $W_{\alpha}$ acting on $\C^n$, and its associated Cherednik algebra.  Let $\Delta^{\alpha}=\Delta_{\mathsf{triv}}$ the standard module corresponding to the trivial module of  $W_{\alpha}$. If $(\dagger)$ holds, $\Delta^{\alpha}$ is projective for all $\alpha\in \Anm$.
\end{lem}
%
%

Recall that Bezrukavnikov and Etingof \cite{bezrukavnikovParabolicInduction2009} defined parabolic induction and restriction functors for rational Cherednik algebras, corresponding to parabolic subgroups.  
Consider the module \[\mathcal{P}=\bigoplus_{\alpha\in \Anm}\Ind{W_{\alpha}}{W_{n}}(\Delta^{\alpha}).\]
\begin{lem}
$\End(\mathcal{P})\cong \Sl^n(m)$.
\end{lem}
\begin{proof}
Note that the image of $\Delta^{\alpha}$ under the $\KZ$ functor for $W_{\alpha}$ is the trivial module over the Hecke algebra $\Hh_{\alpha}$.  
Consider the image $\KZ(\mathcal{P})$.  
By \cite[Cor. 2.3]{ShanCrystal}, we have that 
\[\KZ\circ \Ind{W_{\alpha}}{W_{n}}(\Delta^{\alpha})\cong \Ind{\Hh_{\alpha}}{\Hh_{n}}\circ \KZ_{\alpha}(\Delta^{\alpha})\cong M^{\lambda(\alpha)}=m_{\lambda(\alpha)}\Hh_{n}\]
for all $\alpha \in \Anm$.  Thus, we have an isomorphism $\KZ(\mathcal{P})\cong V^{\otimes n}$ as a module over $\Hh_n$.  Since $ \mathcal{P}$ is projective, this induces an isomorphism
\[\End(\mathcal{P})\cong \Endh(\KZ(\mathcal{P}))\cong \Sl^n(m).\qedhere\]
\end{proof}

Thus, this projective represents the desired functor:
\begin{defn}
	Let $\KZ^B\colon \cO_{h,H}^n\to \Sl^n(m)\mmod$ be the {\bf Schur restriction} functor 
	\[\KZ^B(M)\cong \Hom(\mathcal{P},M)\cong \bigoplus_{\alpha\in \Anm}\Hom(\Delta^{\alpha}, \Res {W_{n}}{W_{\alpha}}(M)).\]
\end{defn}
Equivalently, the algebra $\Sl^n(m)$ appears as a corner algebra of some algebra whose module category is equivalent to $\cO_{h,H}^n$; we can always construct such algebra as endomorphisms of a projective $\mathcal{P}\oplus \mathcal{P}'$ where every indecomposable projective is isomorphic to a summand of $\mathcal{P}$ or $\mathcal{P'}$.  
\begin{rmk}
	Depending on the values $h,H$, we will have a different Verma module corresponding to a 1-dimensional representation which is guaranteed to be projective:
	\begin{enumerate}
		\item If $-rh\notin \Z_{>0}$ and $H-(r-1)h+\frac{1}{2}\notin \Z_{>0}$ then $\Delta_{(1^r),\emptyset}$ is always projective.
		\item If $-rh\notin \Z_{>0}$ and $-H-(r-1)h+\frac{1}{2}\notin \Z_{>0}$ then $\Delta_{\emptyset,(1^r)}$ is always projective.
		\item If $rh\notin \Z_{>0}$ and $-H+(r-1)h+\frac{1}{2}\notin \Z_{>0}$ then $\Delta_{\emptyset,(r)}$ is always projective.
	\end{enumerate}
	Any choice of $h,H$ will be in at least one of these situations for all $m$, based on the sign of the real parts of the parameters.  In each of these cases, we can define a corresponding projective $\mathcal{P}'$ which gives an analogue of $\KZ^B$.  The resulting endomorphism algebra will be $\End(\mathcal{P}')$ with the parameters $(Q,q^{-1}), (Q^{- 1}q^{-1})$ and $(Q^{- 1},q)$ in the 3 situations above, appearing as the endomorphisms of the twists of $M^{\lambda}$'s by the three non-trivial 1-dimensional representations.  
\end{rmk}

\begin{rmk}
	It's worth noting that if we apply the analogous construction in type A, we will obtain a projective $\mathcal{P}^A\cong \oplus_{\la}\Ind{\Hh_\lambda^A}{\Hh_n^A}\Delta^{\alpha}$ whose endomorphisms are a copy of the type A Schur algebra.  This directly realizes the equivalence of \cite[Th. 6.11]{RouqSchur} without using any uniqueness arguments.  This fact has presumably been known to experts for some time, but we could not find a good reference for it.
\end{rmk}

We can extend \cref{th:A} to cover this case as well:
\begin{thm}
If  $(\dagger)$ holds then $\KZ^B$ is a quotient functor, and if $m$ is large, we have a commutative diagram of quotient functors  
	\[\begin{tikzcd}
	& {\Sl^n(m)\mmod} \\
	{\cO_{h,H}^n} && {\Hh_{n}\mmod}
	\arrow["\KZ"', from=2-1, to=2-3]
	\arrow["{\KZ^B}", from=2-1, to=1-2]
	\arrow["Me_{\Hh}", from=1-2, to=2-3]
\end{tikzcd}\]
\end{thm}
\begin{proof}
Each of these functors are the quotient by the subcategory of modules that receive no non-trivial homomorphisms from the representing projective.  The commutation is clear from the fact that:
\[\KZ^B(M)e_{\mathcal{H}}\cong \Hom(e_{\mathcal{H}}\mathcal{P}, M)\cong \Hom(\Ind{1}{W_n}\C_{\mathsf{triv}},M)\cong \KZ(M).\qedhere\]
\end{proof}

This gives us an alternate perspective on the failure of quasi-hereditarity of $\Sl^n(m)$: the algebra $\Sl^n(m)$ cannot be quasi-hereditary if there is a   choice of $h,H$ satisfying $(\dagger)$ such that $\cO_{h,H}^n$ is not equivalent to $\Sd^n(\Lnm)\mmod$ for $m$ large, which manifests in the fact that the quasi-hereditary order for these parameters is not the same as that for $\Sd^n(\Lnm)$, i.e. not a coarsening of the usual dominance order.    
	For example, if $Q=-q^k$ for $0\leq k<n$, then let $h$ be any solution of $q=e^{2\pi i h} $ with negative real part and $H=kh-\frac{1}{2}+a$
	for some integer $a$.  
	This satisfies $(\dagger)$ as long as $a\leq -\Re(kh)$, since in this case   $H+(r-1)h+\frac{1}{2}=(k+r-1)h+a$ for $0<r\leq n$ has negative real part.   
%
If $a<0$, then $((1^n),\emptyset)>((1^{k}),(1^{n-k}))$, as we would expect in the usual dominance order on multipartitions, while if $a\geq 0$, we have $((1^{k}),(1^{n-k}))>((1^n),\emptyset)$.  This matches with the fact that $\Sl^n(m)$ cannot be quasi-hereditary.

On the other hand, if $Q=-q^{k}$ for $0> k>\frac{4-n}{2}$, then there does not seem to be as direct a correspondence between different quasi-hereditary orders and the failure of quasi-hereditarity for $\Sl^n(2n+1)$: if $q$ is not a root of unity, then only one quasi-hereditary is possible satisfying $(\dagger)$, but  $\Sl^n(2n+1)$ will none the less fail to be quasi-hereditary.  

\subsection{Faithfulness}
One of the most important properties a quasi-hereditary cover can have is faithfulness.  Given a quasi-hereditary algebra $A$, we say a quotient functor $F\colon A\mmod \to B\mmod$ is $k$-faithful for $k\in \Z_{\geq 0}$ if for all standard filtered objects $M,N\in \mathcal{C}$, we have an isomorphism
\[\operatorname{Ext}^i_\mathcal{C}(M,N)\cong \operatorname{Ext}^i_{\mathcal{D}}(F(M),F(N))\qquad i=0,\dots, k.\]
Like many other properties, faithfulness is best understood in terms of families.  

Now, assume that $\mathbb{K}$ is a regular local ring with residue field $\Bbbk$ and fraction field $K$.  Assume that we have finite rank  $\mathbb{K}$-algebras $A,B$ and a quotient functor $F:A\mmod \to B\mmod$; that is, $F=\Hom(P,-)$ for a projective $A$-module with $\End(P)=B$ such that:
\begin{enumerate}
	\item $A$ is a finite rank $\mathbb{K}$-quasi-hereditary algebra (i.e. its module category is $\mathbb{K}$-highest weight in the sense of \cite[\S 2.3]{RSVV}).
	\item The base changes $KA,KB$ are semi-simple and $F$ induces a Morita equivalence between them.
\end{enumerate}
 Quasi-hereditarity implies that we have a bijection $\Simp(A)\cong \Simp(KA)$, where every standard $A$-module deforms to a unique simple $KA$-module.  Of course, the Morita equivalence further induces a bijection $\Simp(A)\cong \Simp(KB)$.
 
 \begin{thm}[\mbox{\cite[Th. 4.49]{RouqSchur}}]
 	Assume $A'$ is also a $\mathbb{K}$-quasi-hereditary algebra such that:
 	\begin{enumerate}
 		\item We have a quotient functor $F'\colon A'\mmod \to B\mmod$ satisfying the conditions above.
 		\item The induced bijection $\Simp(A)\cong \Simp(A')$ is order-preserving for some total refinement of the partial orders.
 		\item The functors $F,F'$ are both $1$-faithful
 	\end{enumerate} 
 	Then, there is a unique Morita equivalence between $A$ and $A'$ compatible with $F$ and $F'$:
\[\begin{tikzcd}
	A\mmod && {A'\mmod} \\
	& B\mmod
	\arrow["\sim", from=1-1, to=1-3]
	\arrow["F"', from=1-1, to=2-2]
	\arrow["{F'}", from=1-3, to=2-2]
\end{tikzcd}\]
induced by the bimodules $\Hom_{B}(F(A),F'(A'))$ and $\Hom_{B}(F'(A'),F(A))$.
\end{thm}
The relevant examples for us will be when $\mathbb{K}=\C[[t_0,t_1]]$ with $A$ given by the endomorphisms of a projective generator in $\cO_{h+t_1,H+t_0}^n$, the deformed category $\cO$ with base ring $\mathbb{K}$ (see \cite[\S 6.2]{RSVV}).  We can consider the KZ functor on this deformed category to the deformed Hecke algebra $\Hh_n$ over $\mathbb{K}$ with parameters $(Qe^{2\pi i t_0},qe^{2\pi i t_1})$.  Indeed, this is a quotient functor and an equivalence after base change to the fraction field $K=\operatorname{Frac}(\mathbb{K})$, with the induced bijection $\Simp(\cO^n_{h,H})\cong \Simp(K\Hh_n)$ induced by the known bijection of these sets with bipartitions (since the parameters of the Hecke algebra are generic).  
\begin{thm} 
Assume $n\geq 2$:
	\begin{enumerate}
		\item The category $A\mmod$ is a 1-faithful cover of $\Hh_n$ over $\mathbb{K}$ with parameters $(Qe^{2\pi i t_0},qe^{2\pi i t_1})$ if and only if $Q\neq -1$ and  $q\neq -1$.  
		\item The category $A\mmod$ is a 1-faithful cover of $\Sl^n(m)$ for $m$ large odd over $\mathbb{K}$ with parameters $(Qe^{2\pi i t_0},qe^{2\pi i t_1})$ for all $H,h\in \C$ satisfying $(\dagger)$.
	\end{enumerate}
\end{thm}
\begin{proof}
{\bf Generalities on 1-faithfulness:}	By \cite[Prop. 2.18]{RSVV}, in order to check that $F$ is 1-faithful, it suffices to check that the base change $\Bbbk F$ is 0-faithful, since $KF$ is a Morita equivalence.  By \cite[Lem. 2.8]{RSVV}, this is equivalent to the claim that $\operatorname{lcd}(L)>1$ for all simples with $FL=0$.  
	By \cite[Lem 6.3]{RSVV}, the simples over the Cherednik algebra with $\operatorname{lcd}(L)\leq 1$ are those such that $\Res{W_n}{\langle s\rangle}L\neq 0$ for some simple reflection $s$.  
	
{\bf The case of $\Hh_n$ with $Q,q\neq -1$:}	If $Q\neq -1$, then $\Hh_1$ is semi-simple, and $\Res{\Hh_1}{\Bbbk}$ is conservative.  That is, $\Res{W_n}{\langle s_0\rangle}$ kills a simple if and only if $\KZ$ does.  Additionally, if $q\neq -1$, a similar argument holds for $\langle s_1\rangle$.  Thus, a simple has $\operatorname{lcd}(L)>1$ if and only if it is killed by the $\KZ$ functor, and 0-faithfulness follows.
	
{\bf The case of $\Hh_n$ with $Q=-1$:}	On the other hand, if $Q=-1$, then we can take $A$ to be the quotient of the path algebra of a cyclic quiver below by the relation $xy=t_0$.  The variable $t_1$ does not appear in any relation, so we can simplify our calculations by setting it to 0.   
\[\begin{tikzcd}
	1 && 2
	\arrow["x", curve={height=-18pt}, from=1-1, to=1-3]
	\arrow["y", curve={height=-18pt}, from=1-3, to=1-1]
\end{tikzcd}\]
	This equivalence sends $L_1=L_{(1),\emptyset}$ and $L_2=L_{\emptyset,(1)}$ if $H<0$, and {\it vice versa} if $H>0$.  The projective resolution of $L_1$ over $A$ is given by 
\[\begin{tikzcd}
	{P_1} & {P_2} & {P_1} \\
	& {P_1} & {P_2} & {P_1}
	\arrow["y"', from=1-2, to=1-1]
	\arrow["x"', from=1-3, to=1-2]
	\arrow["{t_0}", from=2-2, to=1-1]
	\arrow["{-1}"{pos=0.1}, shift left=1, from=1-3, to=2-2]
	\arrow["y", from=2-3, to=2-2]
	\arrow["x", from=2-4, to=2-3]
	\arrow["{t_0}", from=2-4, to=1-3]
	\arrow["{t_0}"{pos=0.8}, from=2-3, to=1-2]
\end{tikzcd}\]
This shows that 
\begin{align*}
	\operatorname{Ext}^i(L_1,L_2)&=\begin{cases}
		\Bbbk & i=1,2\\
		0 & i\neq 1,2
	\end{cases}\\
	\operatorname{Ext}^i(L_1,\tilde{L}_2)&=\begin{cases}
		\Bbbk & i=2\\
		0 & i\neq 2
	\end{cases}
\end{align*}
where $\tilde{L}_2$ is the unique deformation of $L_2$ to a free $\Bbbk[t_0]$-module.  The non-zero $\operatorname{Ext}^2$ is a witness to the failure of 1-faithfulness for $\cO_{h+t_1,H+t_0}^1$; the same is true if we tensor both modules with $\C[[t_1]]$.  Now, consider the inductions $ \Ind{\langle s_0\rangle}{W_n}L_1[[t_1]]$ and $\Ind{\langle s_0\rangle}{W_n}\tilde{L}_2[[t_1]]$; the former is still killed by $F$ and the latter is still tilting.  Furthermore, since $\Res{W_n}{\langle s_0\rangle}\Ind{\langle s_0\rangle}{W_n}L_1[[t_1]]$ must be a free $\Bbbk[[t_1]]$-module and is killed by $F$, it must be of the form $(L_1[[t_1]])^{\oplus m}$ for some $m$, with the Mackey formula showing that $m>0$.  Thus we have that 
\begin{multline*}
	\Ext^2(\Ind{\langle s_0\rangle}{W_n}L_1[[t_1]],\Ind{\langle s_0\rangle}{W_n}\tilde{L}_2)\cong \Ext^2(\Res{W_n}{\langle s_0\rangle}\Ind{\langle s_0\rangle}{W_n}L_1[[t_1]],\tilde{L}_2)\\ \cong \Ext^2((L_1[[t_1]])^{\oplus m},\tilde{L}_2)\cong \Bbbk[[t_1]]^{\oplus m}\neq 0
\end{multline*}
This shows that 1-faithfulness fails by \cite[Lem. 2.8]{RSVV}.  

{\bf The case of $\Hh_n$ with $q=-1$:} A similar argument shows the failure of 1-faithfulness if $q=-1$, swapping the roles of $s_0$ and $s_1$, and $t_0$ and $t_1$.  

{\bf The case of $\Sl^n(m)$:} 
For $W_1=\langle s_0\rangle$, the category $\cO$ over the Cherednik algebra has two indecomposable projectives: 
\begin{itemize}
	\item If $Q\neq -1$, these are the two summands of the induction $\Ind{W_0}{\langle s_0\rangle} L_{\emptyset}$, which represents the functor $\KZ$.
	\item If $Q=-1$, then $\Ind{W_0}{\langle s_0\rangle} L_{\emptyset}$ is indecomposable, and the other summand is $\Delta^{(1),\emptyset}$.  
\end{itemize}
Similarly for $W_{(\emptyset,(2))}=\langle s_1\rangle$: If $q\neq -1$, $\Ind{W_0}{\langle s_1\rangle} L_{\emptyset}$ has two summands, while if $q=-1$, it is indecomposable and $\Delta^{\emptyset,(2)}$ is the other indecomposable projective.  

If $L$ is any simple, then by definition if $\KZ^B(L)=0$, then 
\[\Res{W_n}{1}L=0 \qquad \Hom(\Delta^{(1),\emptyset},\Res{W_n}{\langle s_0\rangle}L)=0\qquad \Hom(\Delta^{\emptyset,(2)}, \Res{W_n}{\langle s_1\rangle}L)=0.\]  On the other hand, this means that $\Res{W_n}{\langle s_0\rangle}L=0$ and $\Res{W_n}{\langle s_1\rangle}L=0$, since $\Hom$ of any projective module into these are 0.  This implies that $\operatorname{lcd}(L)>1$ for all $L$ with $\KZ^B(L)=0$.  
\end{proof}